\documentclass[reqno]{amsart}

\usepackage{amsmath,amssymb,amsthm,amsfonts,enumerate, enumitem,hyperref,cleveref,aliascnt}
\usepackage{xcolor}
\usepackage{dsfont}
\usepackage{soul}
\usepackage[all]{xy}
\usepackage[T1]{fontenc}
\usepackage{multirow}
\usepackage{float} % for [H] placement specifier
\usepackage{threeparttable}

\usepackage{pgfplots}
\pgfplotsset{compat=1.15}
\usepackage{mathrsfs}
\usetikzlibrary{arrows}

\DeclareMathOperator{\Span}{Span}

\DeclareMathOperator{\Min}{min}

\DeclareMathOperator{\gn}{gen}
\DeclareMathOperator{\Gn}{Gen}
\DeclareMathOperator{\Ban}{Ban}
\DeclareMathOperator{\GC}{GC}
\DeclareMathOperator{\EC}{EC}

\DeclareMathOperator{\Soc}{soc}
\theoremstyle{plain}

\newtheorem{thrm}{Theorem}[section]

\newaliascnt{cor}{thrm}
\newtheorem{cor}[cor]{Corollary}
\aliascntresetthe{cor}

\newaliascnt{prop}{thrm}
\newtheorem{prop}[prop]{Proposition}
\aliascntresetthe{prop}

\newaliascnt{lem}{thrm}
\newtheorem{lem}[lem]{Lemma}
\aliascntresetthe{lem}

\newaliascnt{conj}{thrm}

\aliascntresetthe{conj}

\theoremstyle{definition}
\newaliascnt{defn}{thrm}
\newtheorem{defn}[defn]{Definition}
\aliascntresetthe{defn}

\newaliascnt{rem}{thrm}
\newtheorem{rem}[rem]{Remark}
\aliascntresetthe{rem}

\newaliascnt{exm}{thrm}

\aliascntresetthe{exm}

\newaliascnt{fact}{thrm}

\aliascntresetthe{fact}

\crefrangeformat{equation}{#3(#1)#4--#5(#2)#6}

\crefname{thrm}{Theorem}{Theorems}
\crefname{theorem}{Theorem}{Theorems}
\crefname{lem}{Lemma}{Lemmas}
\crefname{cor}{Corollary}{Corollaries}
\crefname{prop}{Proposition}{Propositions}
\crefname{defn}{Definition}{Definitions}
\crefname{exm}{Example}{Examples}
\crefname{rem}{Remark}{Remarks}
\crefname{conj}{Conjecture}{Conjectures}
\crefname{quest}{Question}{Questions}
\crefname{section}{Section}{Sections}
\crefname{equation}{\unskip}{\unskip}
\crefname{enumi}{\unskip}{\unskip}
\crefname{subsection}{Subsection}{Subsections}

\newcommand{\lb}{\lambda}

\newcommand{\CC}{\mathbb{C}}

\newcommand{\NN}{\mathbb{N}}

\begin{document}
	\title[Applications of compact multipliers to algebrability]{\texorpdfstring{Applications of compact multipliers to algebrability of $(\ell_{\infty}\setminus c_0)\cup\{0\}$ and $(B(\ell_2)\setminus \mathcal{K}(\ell_2))\cup\{0\}$.}{Applications of compact multipliers to algebrability of (l_infty\c_0)\cup{0} and (B(l_2)\K(l_2))\cup{0}.}}

	\author{W. Franca}
	\address{Departamento de Matem\'atica, ICE - Universidade Federal de Juiz de Fora, 36036-330 Juiz de Fora MG, Brazil.}
	\email{wilian.franca@ufjf.br}

 \author{Jorge J. Garc{\' e}s}
	\address{ Departamento de Matem{\' a}tica Aplicada a la Ingenier{\' i}a Industrial, ETSIDI, Universidad Polit{\' e}cnica de Madrid, Madrid, Spain}
	\email{j.garces@upm.es}
	
	
	\begin{abstract}
We introduce a generator-counting refinement of algebrability for abelian $C^*$-algebras and related Banach algebras. Given an abelian $C^*$-algebra $A$, we define $(C^*)$-genalgebrability in terms of the minimal possible cardinality of a generating set, encoded by the invariants $\gn_{C^*}(A)$ and $\gn(A)$. Using compact multipliers and the ideal $K(A)$ of compact elements, we develop embedding results into $\ell_\infty$ whose ranges avoid $c_0$ (except for the zero vector), and we obtain a universal $(\ell_\infty\setminus c_0)$-embeddability phenomenon under the assumption $K(A)=\{0\}$. As an application, we construct a $^*$-isomorphic copy of $\ell_\infty$ inside $(\ell_\infty\setminus c_0)\cup\{0\}$ and transfer the results to Calkin-type settings such as $(B(\ell_2)\setminus \mathcal K(\ell_2))\cup\{0\}$ and their unitizations. We also establish a generator-counting theorem for abelian $C^*$-algebras: $\gn_{C^*}(A)$ equals the smallest cardinal $n$ for which the spectrum $\Delta(A)$ embeds into $\mathbb R^n$, and we derive topological formulas for $\gn_{C^*}(A)$ in the non-finitely generated case. Finally, we provide a complete classification of the pairs $(d,\kappa)$ for which $(\ell_\infty\setminus c_0)\cup\{0\}$ is $(d,\kappa)$-$(C^*)$-genalgebrable, and we discuss the connection with classical algebrability.
	\end{abstract}
	
	\maketitle
	
	%\tableofcontents
	
\section*{Introduction} 
In recent decades, \emph{lineability} has become a central topic in analysis, particularly in functional analysis. Given a subset $L$ of a vector space $X$---typically consisting of ``pathological'' objects or, more generally, of elements sharing a special property---one asks whether $L\cup\{0\}$ contains a large linear structure, for instance an infinite-dimensional subspace. When such a subspace can be chosen closed in $X$, one speaks of \emph{spaceability}. The modern theory traces back to the seminal result of Gurariy \cite{Gurariy66}, who proved that the set of continuous nowhere differentiable functions on $[0,1]$ contains an infinite-dimensional subspace; later, Fonf et al.~\cite{Fonf1999} strengthened this by showing that the same set even contains an infinite-dimensional \emph{closed} subspace.

Since then, a vast literature has developed. Classical sequence spaces provide one particularly fruitful setting in which to uncover new phenomena in lineability theory (see, for instance, \cite{Aizpuru2006,Araujo2017,BOTELHO2011,Cariello_Sepulveda,FONF201494,accumulation_points,Axarlis2021}), but the scope of the subject is much broader: lineability techniques also interact with many other parts of analysis, including $L_p$-spaces \cite{Botelho2012}, Lebesgue measurable functions \cite{Araujo2017}, classes of linear operators \cite{Diniz2021}, operator ideals in quasi-Banach spaces \cite{Aires2025}, spaces of Lipschitz functions \cite{Aviles2024}, and probabilistic questions such as convergence of random variables \cite{Araujo2024}. Moreover, there are also applications beyond these core areas; for example, \cite[Section~7]{Raposo2023} uses spaceability to study the size and geometry of sets of $m$-linear forms failing Hardy--Littlewood type inequalities for non-admissible exponents. For a panoramic view, we refer to \cite{BPS14} and the monograph \cite{A.B.P.S}.

We would be remiss not to mention that, in recent years, some authors have begun to investigate more restrictive notions of lineability. The motivation is that, in many classical settings, once a set of ``pathological'' objects is known to be nonempty, it often turns out to contain unexpectedly large linear structures; thus genuinely natural examples of non-lineable sets can be difficult to produce. This scarcity has led to refined concepts designed to measure linear structure more delicately and to provide a meaningful hierarchy between mere nonemptiness and the existence of large subspaces (see the introduction of \cite{EmerickFranca2026} and the references therein). In this paper, we contribute to this refinement program by introducing a generator-counting variant of algebrability.

Historically, the terms \emph{lineable} and \emph{spaceable} were coined by Aron, Gurariy, and Seoane-Sep\'ulveda \cite{AGS05} (see also Enflo, Gurariy, and Seoane-Sep\'ulveda \cite{EGS14}). The stronger notion of \emph{algebrability} was introduced by Aron, P\'erez-Garc\'ia, and Seoane-Sep\'ulveda \cite{APS06} and later refined by Bayart--Quarta \cite{BQ07} and Bartoszewicz--G\l\k{a}b \cite{BG13}; see also \cite[Definition~V.3]{A.B.P.S} for further variants and Definition~\ref{def classical algebrable} for a concrete example.
To date, the literature on algebrability has largely focused on the cardinality of an infinite minimal generating set of the algebra under consideration. From this standpoint, however, technical difficulties can arise: an algebra may fail to admit a minimal generating set, and even when it does, minimal generating sets can have different cardinalities.

Although our results are connected to the classical notion of algebrability, the methods developed in this paper allow us to control the number of generators of a given abelian C$^*$-algebra $A$. For this reason, we work with a variant of algebrability, since the standard Definition~\ref{def classical algebrable} is phrased in terms of infinite minimal generating sets.
More precisely, we introduce \emph{$($C$^{*})$-genalgebrability} (Definitions~\ref{def genalgebrable} and~\ref{defn.alg.}), defined via the smallest possible cardinality of a generating set for $A$, either as a C$^*$-algebra or as an algebra---denoted by $\gn_{C^*}(A)$ and $\gn(A)$, respectively. These cardinals always exist (see \cref{rem Gen c* are sets}) and the definition does not rely on the existence of a minimal generating set. This point is essential because we treat a broad class of abelian C$^*$-algebras for which the existence of minimal generating sets is, in general, delicate.

The paper is organized as follows.

\Cref{Section 1} fixes notation and records preliminary facts about compact elements in a Banach algebra (that is, those $a$ for which the operator $U_a(x)=axa$ is compact). More concretely, we characterize the existence of nonzero compact elements in a semisimple commutative Banach algebra in terms of the existence of isolated points in its Gelfand spectrum. The ideal of compact elements of a Banach algebra (denoted by $K(A)$) plays a crucial role in \Cref{section 2}.

The second section is divided into three acts. We begin \cref{section 2} by studying which uniform algebras embed into $\ell_\infty$. For an abelian $C^*$-algebra $A$, this is completely characterized by separability of its spectrum $\Delta(A)$ (\cref{lem linf-embedable}), whereas for general uniform algebras we obtain only a sufficient condition.

The second step is to study when a uniform algebra (and, more generally, an algebra or Banach space) admits an embedding into $\ell_\infty$ whose range is disjoint from $c_0$ except for the zero vector. Using that $K(\ell_\infty)=c_0$, we show in \cref{main lineab banch algebra1} that whenever a uniform algebra $A$ embeds (in the appropriate sense) into $\ell_\infty$ and $K(A)=\{0\}$, there exists an embedding $\phi$ such that $\phi(A)\cap c_0=\{0\}$. Moreover, the assumption $K(A)=\{0\}$ forces every embedding $\phi\colon A\to\ell_\infty$ to satisfy $\phi(A)\cap c_0=\{0\}$. We call this phenomenon \emph{universal $(\ell_{\infty}\setminus c_{0})$-embeddability}.

We then use a simple topological argument to show that if $C_0(L)$ is universally $(\ell_{\infty}\setminus c_{0})$-embeddable, then so is $C_b(L)$ (\cref{thm extend algebrability to M(A)}). In \cref{trm l_inf in l_inf--c0}, as an application of \cref{thm extend algebrability to M(A)}, we build a $^*$-isomorphic copy of $\ell_\infty$ inside $(\ell_{\infty}\setminus c_{0})$ and deduce the following: every algebra, uniform algebra, or Banach space that embeds into $\ell_\infty$ also admits an embedding that avoids $c_0$ (see also the comments after \cref{trm l_inf in l_inf--c0}).

The third act transfers our results from $(\ell_{\infty}\setminus c_{0})\cup\{0\}$ to other settings. First, using the \emph{diagonal} embedding, we pass to $(B(\ell_2)\setminus K(\ell_2))\cup\{0\}$ (\cref{cor abelian triple in B(H)}).
Next, we treat the corresponding \emph{unitized} settings: we show that every algebra, uniform algebra, or Banach space that embeds into $\ell_\infty$ also embeds into $\ell_\infty$ and into $B(\ell_2)$ in a way that avoids $c$ and $K(\ell_2)\oplus \CC 1_{B(\ell_2)}$ (\cref{thrm l-c embedd} and \cref{cor B-(K+c) emebd}). Under an additional assumption, we also obtain universal embeddability results for these sets ( see also \cref{can not be finitely generated}).

The first part of \cref{section 3} studies the quantities $\gn(A)$ and $\gn_{C^*}(A)$. We show that if $A$ is an infinitely generated C$^*$-algebra, then $\gn_{C^*}(A)$ coincides with the smallest cardinal $n$ such that its spectrum $\Delta(A)$ admits an embedding into $\mathbb{R}^{n}$ (\cref{teo genCstar}); the finitely generated version of this result is due to Nagisa \cite[Proposition 2]{Nagisa}. Next, we give a formula for $\gn_{C^*}(A)$ in terms of topological invariants, namely the weight of $\Delta(A)$ and the density character of $A$ (\cref{prop num gen C*}).

Second, combining these results with suitable choices of compact Hausdorff spaces $K_\kappa$ and the results of \cref{section 2}, we give in \cref{teo n-alg and n-ctar-alg} a complete answer to when $(\ell_{\infty}\setminus c_0)\cup\{0\}$ contains a $^*$-isomorphic copy of a C$^*$-algebra $A$ with $d=\dim(A)$ and $\gn_{C^*}(A)=\kappa$ (that is, when it is $(d,\kappa)$-C$^*$-genalgebrable). Notably, we do so without assuming the existence of minimal generating sets. Moreover, \cref{teo n-alg and n-ctar-alg} includes a purely algebraic counterpart, based on $\gn(A)$ instead of $\gn_{C^*}(A)$. We obtain the following complete classification:
Let $M=(\ell_{\infty}\setminus c_0)\cup \{0\}$. Then:
\begin{enumerate}
    \item $M$ is $(d,1)$-genalgebrable and $(d,1)$-C$^*$-genalgebrable for every $d\in \mathbb{N}$.
    \item  If $d,n<\infty$ and $n\neq 1$, then $M$ is neither $(d,n)$-genalgebrable nor $(d,n)$-C$^*$-genalgebrable.
    \item  $M$ is $(\mathfrak{c},\kappa)$-C$^*$-genalgebrable for every $\kappa\leq \mathfrak{c}$. 
    \item  $M$ is $(\aleph_0,\kappa)$-genalgebrable for every $\kappa\leq \aleph_0$.
    \item  $M$ is $(\kappa,\kappa)$-genalgebrable for every $\aleph_0 < \kappa \leq \mathfrak{c}$. 
    \end{enumerate}

It is important to note that the results above, stated in terms of $\gn(A)$ and $\gn_{C^*}(A)$, do not automatically transfer to classical $(\alpha,\beta)$-algebrability (or to $(\alpha,\beta)$-C$^*$-algebrability, its C$^*$-algebraic analogue introduced in \cref{missing}), since when these invariants are infinite the existence of a minimal generating set is not guaranteed. Nevertheless, for the specific algebras used in the proof of \cref{teo n-alg and n-ctar-alg} we can compute minimal generating sets, which allows us to transfer the conclusions to $(\alpha,\beta)$-algebrability and $(\alpha,\beta)$-C$^*$-algebrability, except for those pairs $(d,n)$ with $d<\infty$ and $n\in\NN\setminus\{1\}$, where the statements differ; in this regime the answer becomes affirmative (see \cref{2-teo n-alg and n-ctar-alg}). This happens because minimal generating sets of different cardinalities may exist in that case.

Finally, \cref{cor c*algebrable} shows that every affirmative algebrability result for $(\ell_{\infty}\setminus c_0)\cup\{0\}$ (in any of the senses considered) also holds for the sets
$$((B(\ell_2)\setminus \mathcal{K}(\ell_2))\cup \{ 0\}, (\ell_{\infty}\setminus c)\cup \{0\}, (B(\ell_2)\setminus (\mathcal{K}(\ell_2) )\oplus 1_{B(\ell_2)}))\cup \{0\}.$$

\section{Notation and a preliminary result}\label{Section 1}
For Banach spaces $X$ and $Y$, we write $B(X,Y)$ for the space of all bounded linear operators from $X$ to $Y$. It is well known that $B(X,Y)$ is a Banach space when endowed with the operator norm
$$\|T\|:=\sup_{x\in B_X}\|T(x)\|,\qquad B_X:=\{x\in X: \|x\|\le 1\}.$$
As usual, we write $\mathcal{K}(X,Y)$ for the set of compact operators in $B(X,Y).$ When $X=Y$, we adopt the conventions $B(X):=B(X,X)$ and $\mathcal{K}(X,X):=\mathcal{K}(X)$; these are then Banach algebras.

Throughout this work all the algebras are associative. Let $A$ be a Banach algebra.
For each $a\in A$, we define $L_a\in B(A)$ by $L_a(x):=ax.$ Operators of the form $L_a$ are called (left) multipliers. It is straightforward to check that $\|L_a\|=\|a\|$ whenever $A$ is unital (and $\|1_A\|=1$). We say that $a$ is a \emph{compact} (respectively, \emph{weakly compact}) left multiplier if $L_a$ is a compact (respectively, weakly compact) operator. Compact and weakly compact right multipliers are defined analogously via $R_a(x)=xa$. According to the introduction of \cite{Ylinen_wcc}, if $a$ is a weakly compact left multiplier in a C$^*$-algebra $A$, then $a$ is also a weakly compact right multiplier, and conversely.

Following \cite{AlexanderCompact}, we say that $a$ is \emph{compact} when the operator $U_a\colon A\to A$, $x\mapsto axa$, is compact. We denote by $K(A)$ the set of compact elements of $A$. By \cite[Theorem 3.1]{Ylinen_wcc}, if $A$ is a C$^*$-algebra, then $a\in A$ is compact if and only if $a$ is a weakly compact left (equivalently, right) multiplier. It is also known that $K(A)$ is a norm-closed, two-sided ideal of $A$ (see the introduction of \cite{Ylinen_wcc}). Actually, it is not hard to see that this conclusion is true for Banach algebras in general. We now show that bounded isomorphisms between Banach algebras preserve left/right compact multipliers and compact elements.

\begin{prop} \label{compact.operator}
   Let  $\Phi:A\to B$ be a bounded isomorphism of Banach algebras. Fix $a\in A.$ Then the following assertions hold:
   \begin{enumerate}
       \item \label{preserve compact 1} $a$ is a compact element in $A$ if and only if $\Phi(a)$ is a compact element in $B.$
      \item \label{preserve compact 2}  $a$ is a left (right) compact multiplier in $A$ if and only if $\Phi(a)$ is a left (right) compact multiplier  in $B.$
   \end{enumerate}
\end{prop}
   \begin{proof}
We prove only $(i)$, since the proof of $(ii)$ is analogous.
Let us suppose that $a$ is a compact element in $A$. Let $(b_k)_{k\in\mathbb{N}}$ be a bounded sequence in $B.$ Since $\Phi^{-1}$  is also continuous, then $(c_k)_{k\in\mathbb{N}}=(\Phi^{-1}(b_{k}))_{k\in\mathbb{N}}$ is a bounded sequence in $A.$ From the compactness of the operator $U_a,$ we  conclude that $(U_a(c_k))_{k\in\mathbb{N}}$ contains a convergent subsequence in $A.$ We can assume without loss of generality that the sequence $(U_a(c_k))_{k\in\mathbb{N}}$  converges to some $c\in A.$ Hence, $\Phi(U_a(c_k) )=U_{\Phi(a)} b_k$ converges to $\Phi(c),$  showing that $ \Phi(a)$ is a compact element in $B.$ The converse follows by applying the same argument to $\Phi^{-1}$.
\end{proof}

Let $L$ be a locally compact Hausdorff space, and let
$$C(L)=\{f\colon L\to\mathbb{C}: f \ \mbox{is continuous}\}.$$
We say that $f\in C(L)$ \emph{vanishes at infinity} if, for every $\epsilon>0$, there exists a compact set $K=K(\epsilon)\subset L$ such that $|f(x)|<\epsilon$ whenever $x\notin K.$ We write $C_0(L)$ for the set of all functions in $C(L)$ that vanish at infinity. It is well known that $C_0(L)$, equipped with pointwise addition and multiplication and the sup norm, is a commutative C$^*$-algebra; here $f^{*}(t)=\overline{f(t)}$ for each $t\in L$ and $f\in C(L).$ In particular, $C_0(L)$ is non-unital. We denote by $C_b(L)$ the algebra of bounded continuous functions on $L$, which is a unital C$^*$-algebra. If $K$ is compact Hausdorff, then $C(K)$ is also a unital C$^*$-algebra.

In this paragraph, let $A$ be a commutative Banach algebra. The set
$$\Delta(A)=\{\phi\in A^{*}\setminus\{0\}: \phi \ \mbox{is multiplicative}\}$$
is called the spectrum (also known as the structure space or maximal ideal space) of $A.$ We regard $\Delta(A)$ as a topological space endowed with the weak$^{*}$-topology. By \cite[Proposition 17.2]{BD}, $\Delta(A)$ is locally compact and Hausdorff (and compact if $A$ is unital).
Moreover, if $A$ is a semisimple commutative Banach algebra, then the Gelfand transform
\begin{align*}
    \hat{}\colon A &\to C_0(\Delta(A))\\
    a&\mapsto \hat{a}\colon \Delta(A)\to \mathbb{C}, \quad \text{where } \hat{a}(\phi)=\phi(a),
\end{align*}
is a contractive injective homomorphism (see \cite[Corollary 17.7]{BD}). However, it need not be isometric, nor must it have closed range, even in the semisimple case. By the uniqueness of the norm topology on a semisimple Banach algebra (see \cite{Johnson}), the Gelfand transform has closed range if and only if there exists $k>0$ such that $k\|a\|\le \|\hat{a}\|$ for all $a\in A$.

It is also well known that the Gelfand transform is isometric if and only if $\|a^2\|=\|a\|^2$. Throughout the paper, we call a commutative Banach algebra satisfying this identity a \textit{uniform algebra}, even when it is not unital. If $A$ is an abelian C$^*$-algebra, then the Gelfand transform is a $^*$-isomorphism. In this setting, the existence of compact elements in a commutative C$^*$-algebra is characterized by the existence of isolated points in $\Delta(A)$. As we see below, the same is true for semisimple Banach algebras.

\begin{lem}\label{lem no compact elements}
    Let $A$ be a  semisimple commutative Banach algebra. Then $K(A)=\{0\}$ if and only if $\Delta(A)$ has no isolated points.   
 \end{lem}
\begin{proof} First we assume that $\Delta(A)$ has no isolated points.
By \cite{Kamowitz}, $A$ has no nonzero compact left (and hence right) multipliers. If $a\in A$ is a compact element, then $a^2$ is a compact left multiplier, and hence $a^2=0$. It follows that $\hat{a}(t)=0$ for all $t\in\Delta(A)$, so $\hat{a}=0$ and, by injectivity of the Gelfand transform, $a=0$. Thus $K(A)=\{0\}$.

Conversely, assume that $\Delta(A)$ has an isolated point. By the comments following the proof of the main theorem in \cite{Kamowitz}, $A$ admits a compact multiplier $L_a$ with $a\neq 0$. Then $U_a=L_a\circ R_a$ is compact, so $0\neq a\in K(A)$.
\end{proof}

By \cite[Corollary 17.7]{BD}, every uniform algebra is semisimple; hence \cref{lem no compact elements} applies in particular to uniform algebras.
 
Now, let $X$ be a separable Banach space. By \cite[Proposition 5.11]{F.H.H.P.Z}, there exists a sequence $(f_n)\subset X^*$ such that
\begin{align*}
    \|x\|=\sup_n \{|f_n(x)|\}
\end{align*}
for every $x\in X.$ Consequently, the linear map
\begin{equation}\label{eq countable norming set 2}
\begin{split}
T\colon X&\to \ell_{\infty}\\
x&\mapsto T(x)=(f_n(x))
\end{split}
\end{equation}
is an isometric isomorphism from $X$ onto a closed subspace of $\ell_{\infty}$. Although such an embedding always exists when $X$ is separable, we do not control the position of $T(X)$ inside $\ell_{\infty}.$ In this paper, we address this issue for certain uniform algebras, with particular attention to abelian C$^*$-algebras. 

\begin{defn}
Let $B$ be a C$^*$-algebra and $A$ be a vector space.
We say that an injective linear map $\phi:A\to B$ is a $B$-embedding of $A$ if $\phi$ preserves the relevant structure on $A$, for example:
\begin{itemize}
    \item $\phi$ is isometric, if $A$ is a Banach space;
    \item $\phi$ is an algebraic homomorphism, if $A$ is an algebra;
    \item $\phi$ is an isometric and algebraic homomorphism, if $A$ is a Banach algebra;
    \item$\phi$ is a $^*$-homomorphism, if $A$ is a C$^*$-algebra.
\end{itemize}
In this context, we say that $A$ is $B$-embeddable if there exists a $B$-embedding of $A$. 

We also use the term `copy' and the symbol $\cong$ following the same conventions as above.
\end{defn}

\begin{defn}
Let $A$ be a vector space and let $B$ be a C$^*$-algebra.
Let $X,Y\subset B.$ We say that $X$ \textit{avoids} $Y$ if $X\cap Y=\{0\}.$
We say that a $B$-embedding $\phi:A\to B$ avoids $Y$ if $\phi(A)$ avoids $Y.$
\end{defn}

We consider the following property.

\begin{defn}
  Let $A$ be a vector space, let $B$ be a C$^*$-algebra, and let $Y\subset B$.
   We say that $A$ is $(B\setminus Y)$-embeddable if there exists a $B$-embedding of $A$ that avoids $Y$. We say that $A$ is universally $(B\setminus Y)$-embeddable if $A$ is $(B\setminus Y)$-embeddable and every $B$-embedding of $A$ avoids $Y$. 
  When $Y$ is a vector space, we use the same terminology, even though $0$ belongs to both $Y$ and the embedded copy of $A$ in $B$.
\end{defn}

In the next section, we determine which algebras are $(\ell_{\infty}\setminus c_{0})$-embeddable and, among them, characterize those that are universally $(\ell_{\infty}\setminus c_{0})$-embeddable. We then use the \textit{diagonal embedding} to transfer these results to the noncommutative analogues of these sets.

\section{\texorpdfstring{\( (\ell_{\infty}\setminus c_{0}) \)-embeddable uniform algebras}{ell_infty minus c_0-embeddable uniform algebras}\label{section 2}}

As a first step in studying $(\ell_{\infty}\setminus c_{0})$-embeddability, we need to understand which algebras are $\ell_{\infty}$-embeddable. For abelian C$^{*}$-algebras, this can be characterized in terms of the Gelfand spectrum. Let $A$ be a commutative C$^{*}$-algebra and let $\phi:A\to\ell_{\infty}$ be an injective $^{*}$-homomorphism. For each $n$, let $\delta_n\in \Delta(\ell_\infty)$ denote the evaluation functional at the $n$th coordinate. Since $\phi$ is an injective $^*$-homomorphism, $\psi_n=\phi^*(\delta_n)=(\delta_n\circ\phi)\in \Delta(A)$ for all $n$. Moreover, we have:
$$\|a\|=\|\phi(a)\|=\sup_n |\phi(a)_n|=\sup_n |\delta_n(\phi(a))|=\sup_n |\psi_n(a)| .$$ We claim that
$R=\{ \psi_n:n\in \NN\}$ is dense in $\Delta(A).$ If there exists $x_0\in \Delta(A)\setminus \overline{R}$, then we can find $f\in C_0(\Delta(A))$ such that $f(x_0)=1$ and $f(R)=0.$ By Gelfand theory, there exists a unique $a\in A$ such that $\hat{a}=f.$ Since $f\neq 0$, we have $a\neq 0.$ On the other hand,
$$\|a\|=\sup_n |\psi_n(a)| =\sup_n |\hat{a}(\psi_n)|=0 ,$$
a contradiction. We have shown that if $A$ is an $\ell_{\infty}$-embeddable C$^{*}$-algebra, then $\Delta(A)$ is separable. Conversely, if $\Delta(A)$ is separable, take a dense subset $R=\{x_n: n\in\mathbb{N}\}\subseteq\Delta(A)$ and define
$T:C_0 (\Delta(A)) \to \ell_{\infty}$ by $T(f)=(f(x_n))_n.$ Clearly, $T$ is an injective $^{*}$-homomorphism between C$^{*}$-algebras and hence an $\ell_{\infty}$-embedding.

\begin{lem}\label{lem linf-embedable}
Let $A$ be an abelian Banach algebra. If $A$ is a C$^{*}$-algebra or a uniform algebra and $\Delta(A)$ is separable, then $A$ is $\ell_{\infty}$-embeddable. Moreover, if $A$ is a C$^{*}$-algebra, then the converse also holds.
\end{lem}
\begin{proof}
    The case in which $A$ is a C$^{*}$-algebra is standard; we included the details above for completeness. Now suppose that $A$ is a uniform algebra and $\Delta(A)$ is separable. Let $R=\{x_n\mid n\in\mathbb{N}\}$ be a dense subset of $\Delta(A).$ Let us define 
$T:C_0 (\Delta(A)) \to \ell_{\infty} ,T(f)=(f(x_n)).$ Clearly, $T$ is an injective $^*$-homomorphism between C$^*$-algebras and hence it is isometric.
Now define $G:A\to \ell_{\infty}$ given by $a\mapsto G(a)= T(\hat{a}),$ which is clearly an isometry (since $A$ is uniform) and a homomorphism. If $A$ is a C$^*$-algebra then the Gelfand transform is a $^*$-isomorphism and hence $G$ is an injective $^*$-homomorphism.
\end{proof}

\begin{rem}\label{snark comment}
The reader may wonder why \cref{lem linf-embedable} does not yield an equivalence for uniform algebras. The point is that, if $A$ is a closed subalgebra of $\ell_{\infty}$, there is a natural continuous map $\phi^*:\beta\NN\to\Delta(A)$, but it need not be surjective; hence, unlike the C$^{*}$-subalgebra case, separability of $\Delta(A)$ is not automatic (unless $A$ satisfies the Corona Theorem \cite{Gelfan_duality_NOT}).
\end{rem}

Commutative Banach algebras with separable spectrum include, in particular, all separable semisimple Banach algebras, as the following lemma shows (see also \cref{revisisting}).

\begin{lem}\label{lem Gelfand separable}
    Let $A$ be a commutative semisimple Banach algebra. Then $A$ is separable if and only if $C_0(\Delta(A))$ is separable. 
\end{lem}
\begin{proof}
   If $A$ is separable, then we may derive that $B_{A^*}$ is compact and metrizable (in the weak$^*$-topology) hence $B_{A^*}$ is separable. As a consequence,  $\Delta(A)$ is locally compact, metrizable and separable (because it is a subspace of a separable metric space). Thus,   $C_0(\Delta(A))$ is separable (see the second page in the introduction of \cite{C.Y.Chou}).

   The other implication is clear since, by applying the Gelfand transform, we can see $A$ as a subspace of the separable (metric) space $C_0(\Delta(A)).$  
\end{proof}

\begin{rem}\label{revisisting}
As a consequence of \cref{lem Gelfand separable} and \cite[Theorem 2.4]{C.Y.Chou}, a commutative semisimple Banach algebra $A$ is separable if and only if $\Delta(A)$ is metrizable and separable if and only if $\Delta(A)$ is $\sigma$-compact and metrizable if and only if $\Delta(A)$ is second countable. A typical example of nonseparable abelian C$^*$-algebra with separable spectrum is $\ell_\infty$, where $\Delta(\ell_\infty)=\beta \NN$ is separable but not metrizable.
\end{rem}

The next result shows that universal $(\ell_{\infty}\setminus c_{0})$-embeddability fails exactly when $A$ has a nonzero compact element.

\begin{thrm}\label{main lineab banch algebra1}
Let $A$ be an $\ell_\infty$-embeddable uniform algebra (or an abelian C$^*$-algebra). 
Then $A$ is universally  $(\ell_{\infty}\setminus c_{0})$-embeddable if and only if $K(A)=\{ 0\}.$
\end{thrm}
\begin{proof}
 Suppose first that $K(A)=\{0\}$ and let $\phi:A \to \ell_\infty$ be an embedding. It is well known that $K(\ell_{\infty})=c_{0}.$ Let us suppose that there exists $z\in \phi(A)\cap c_0$ with $z\neq 0.$ Observe that $\phi$ is an isometric homomorphism or an injective $^*$-homomorphism. In the latter case, $\phi$ is also isometric (actually, every injective $^*$-homomorphism between C$^*$-algebras is isometric). Since $\phi(A)$ is a closed subalgebra of $\ell_{\infty}$ then $U_z(\phi(A))\subseteq \phi(A).$ It is clear that the restriction of $U_z$ to $\phi(A)$ is a compact operator on $\phi(A)$, that is, $z$ is a compact element in $\phi(A)$. Since $\phi(A)$ is itself a Banach algebra and $\phi^{-1}\in B(\phi(A),A)$ (by the Open mapping Theorem), then from \cref{compact.operator} we see that $\phi^{-1}(z)\in K(A)=\{0\}$ and hence $z=0$.

 For the converse, assume that $K(A)\neq \{0\}$. We  construct an $\ell_\infty$-embedding whose range meets $c_0$. By \cref{lem no compact elements}, $\Delta(A)$ contains an isolated point $x_0$. By the \v{S}ilov idempotent theorem there exists an idempotent $e\in A$ such that $\hat e=\chi_{\{x_0\}}$ (in particular, $\hat e(x_0)=1$ and $\hat e(y)=0$ for all $y\in\Delta(A)\setminus\{x_0\}$).
 
 Define
 $$
 I_{x_0}:=\{\hat a\in \hat A : \hat a(x_0)=0\}.
$$
 Then $\hat A= I_{x_0}\oplus \CC\hat e$. Indeed, for every $\hat a\in \hat A$ we may write
$$
 \hat a=(\hat a-\hat a(x_0)\hat e)+\hat a(x_0)\hat e,
$$
 where $\hat a-\hat a(x_0)\hat e\in I_{x_0}$; moreover,
$$(\hat a-\hat a(x_0)\hat e)\,\hat e=0.$$
 Consequently, the two summands have disjoint supports, and hence
  $$
  \|\hat{a}\|=\|(\hat{a}-\hat{a}(x_0)\hat{e})+\hat{a}(x_0)\hat{e}\|_\infty=\max \{ \|\hat{a}-\hat{a}(x_0)\hat{e}\|_\infty,|\hat{a}(x_0)|\}.$$  Pulling back the Gelfand transform we have the decomposition $A=I\oplus \CC e,$ (where $\hat{I}=I$). Now every $a\in A$ writes in the form $(a-\hat{a}(x_0)e)+\hat{a}(x_0)e.$ Additionally, since the Gelfand transform is an isometric homomorphism, we have 
 $$(a-\hat{a}(x_0)e) \cdot e=0 \mbox{ and } \|a\|=\max\{\|a-\hat{a}(x_0)e\|, |\hat{a}(x_0)|\}.$$ Let $\phi:A \to \ell_\infty$ be an $\ell_\infty$-embedding. For each $a\in A$ we write $\phi(a)=(\phi(a)_n)$. Next, we define 
 $$\psi(a)=(\hat{a}(x_0),\phi(a-\hat{a}(x_0) e)_1,\phi(a-\hat{a}(x_0) e)_2,\dots).
 $$
We claim that $\psi$ is multiplicative. Indeed, we recall 
$A = I \oplus \CC e$ with $Ie=\{0\}
$. We write 
$a=a_0+\alpha e$ and  $b=b_0+\beta e$
where
$
a_0:=a-\hat a(x_0)e$ and $b_0:=b-\hat b(x_0)e\in I$
$
\alpha=\hat a(x_0),\ \beta=\hat b(x_0).$
Then $a_0e=b_0e=0.$ So,
$
ab=(a_0+\alpha e)(b_0+\beta e)=a_0b_0+\alpha\beta e.
$
In addition, 
$
\widehat{ab}(x_0)=\hat a(x_0)\hat b(x_0)=\alpha\beta.
$
Hence, $
ab-\widehat{ab}(x_0)e = a_0b_0.
$
Finally,
\begin{align*}
\psi(ab)_0 & = \widehat{ab}(x_0)=\alpha\beta=\psi(a)_0\,\psi(b)_0,\\
\psi(ab)_n & = \phi(a_0b_0)_n=\phi(a_0)_n\,\phi(b_0)_n=\psi(a)_n\,\psi(b)_n.
\end{align*}
Therefore $\psi(ab)=\psi(a)\psi(b)$ for all $a,b\in A$.

Hence,  $\psi$ is an (injective) homomorphism ($^*$-homomorphism in case $A$ is a C$^*$-algebra), then $\psi$ is a new embedding. Moreover, 
$$\psi(a)=\max \{|\hat{a}(x_0)|,\| \phi(a-\hat{a}(x_0) e)\|_{\infty}\}=\max \{|\hat{a}(x_0)|,\| a-\hat{a}(x_0) e\|\}   =\|a\|$$
thus $\psi$ is an isometric homomorphism. At last, observe that 
$$\psi(e)=(1,0,0,\dots)\in c_0,$$
completing the proof.
\end{proof}

In view of \cref{main lineab banch algebra1}, we obtain the following corollary.

\begin{cor}
   Let $A$ be a $\ell_\infty$-embeddable uniform algebra or abelian C$^*$-algebra %\st{with $\Delta(A)$ separable}
   and $1\leq p<\infty.$ Then $A$ is $(\ell_{\infty}\setminus \ell_{p})$-embeddable. Moreover, $A$ is universally  $(\ell_{\infty}\setminus \ell_{p})$-embeddable if and only if $K(A)=\{ 0\}.$ 
\end{cor}

\subsection{\texorpdfstring{\((\ell_{\infty}\setminus c_{0})\)-embeddablility of mutipliers}{M(A)}}

Let $A$ be a commutative C$^*$-algebra. A map $T\colon A\to A$ is called a \textit{multiplier} if $T(ab)=aT(b)$ for all $a,b\in A.$ We denote by $M(A)$ the \textit{multiplier algebra} of $A$, that is,
$$
M(A)=\{T\in B(A): T\ \mbox{is a multiplier}\}.
$$
In case $A$ is unital then $M(A)=A.$ Otherwise, by  \cite{Busby68}[Theorem 2.11 and Proposition 3.1]  $M(A)$ is an abelian C$^*$-algebra and there is canonical embedding that identifies $A$ with a C$^*$-subalgebra (actually, as an ideal) of $M(A)$. It is well known that if $A=C_0(L)$ then $M(A)=C_b(L)\cong C(\beta L)$ (where $\beta L$ stands for the Stone--Čech compactification of $L$). Thus, if $A$ is a non-unital commutative C$^*$-algebra, so that $A\cong C_0(\Delta(A))$, then $M(A)\cong C_b(\Delta(A))\cong C(\beta\Delta(A)).$ In particular, if $\Delta(A)$ is separable then $\beta\Delta(A)$ is also separable, since $\Delta(A)$ is dense in $\beta\Delta(A).$ Moreover, the isolated points of $\beta\Delta(A)$ (if any) coincide with the isolated points of $\Delta(A)$ (because $\Delta(A)$ is open and dense in $\beta\Delta(A)$).
Consequently, $K(A)=\{0\} \Leftrightarrow K(M(A))=\{0\}$ by \cref{lem no compact elements}.

If $A$ is unital, then $M(A)=A$, so the conclusion holds in this case as well.
\begin{cor}\label{thm extend algebrability to M(A)}
   Let $A$ be a $\ell_\infty$-embeddable  
   %\st{non-unital} 
   abelian C$^*$-algebra. Then $M(A)$ is $\ell_{\infty}$-embeddable. 
     Moreover, $M(A)$ is universally $(\ell_{\infty}\setminus c_{0})$-embeddable if and only if $A$ is universally $(\ell_{\infty}\setminus c_{0})$-embeddable.
\end{cor}
\begin{proof}
    The result follows from \cref{lem linf-embedable}, \cref{main lineab banch algebra1} and the comments above.
\end{proof}

So far we know that every $\ell_\infty$-embeddable uniform algebra with $K(A)=\{0\}$ is $(\ell_{\infty}\setminus c_{0})$-embeddable. Surprisingly, we obtain the strongest possible statement: every subalgebra (and every closed subspace) of $\ell_{\infty}$ can be relocated to avoid $c_{0}$.
\begin{cor}\label{trm l_inf in l_inf--c0}
$\ell_\infty$ is $(\ell_{\infty}\setminus c_{0}) $-embeddable. As a consequence, every $\ell_\infty$-embeddable algebra and every $\ell_\infty$-embeddable Banach space is $(\ell_{\infty}\setminus c_{0}) $-embeddable. In particular, every subalgebra and every closed subspace of $\ell_\infty$ is $(\ell_{\infty}\setminus c_{0}) $-embeddable. 
\end{cor}
\begin{proof}
  Let $(I_{j})$ be a sequence of disjoint (non-degenerated) closed intervals on $\mathbb{R}.$  We define $\displaystyle{L=\bigcup_{j=1}^{\infty}I_j.}$ It is clear that $L$ is separable, locally compact Hausdorff and has no isolated points (so $K(C_0(L))=\{0\}$ by \cref{lem no compact elements}). 
 
 By \cref{lem linf-embedable} and \cref{main lineab banch algebra1} $C_0(L)$ is universally $(\ell_{\infty}\setminus c_{0})$-embeddable. Thus,  we can find a $^*$-isomorphic copy $B$ of $C_b(L)=M(C_0(L))$ inside $   
 M$ by  \cref{thm extend algebrability to M(A)}. Consequently,  there exists a $\ell_{\infty}$-embedding $\phi:C_b(L)\to \ell_\infty$ 
 that avoids $c_0$. Clearly, $C_b(L)$ contains a $^*$-isomorphic copy of $\ell_\infty.$ More concretely, the mapping $T: \ell_\infty \to C_b(L)$ given by $T((\lambda_n))=\sum_{n=1}^{\infty}\lambda_n \chi_{I_n}$ is an injective $^*$-homomorphism (the given sum is only formal, and that being said we have $T((\lambda_n))\cdot T((\beta_n))=\sum_{n=1}^{\infty}\lambda_n \chi_{I_n}\cdot \sum_{n=1}^{\infty}\beta_n \chi_{I_n}=\sum_{n=1}^{\infty}\lambda_n\beta_n \chi_{I_n}=T((\lambda_n)(\beta_n))$).
 Then $A_\infty:=\phi(T( \ell_\infty) )$ is the desired copy of $\ell_\infty$ that avoids $c_0.$

 We have just proved that there exists a $^*$-isomorphism $\psi_1:A_\infty \to \ell_\infty$ and an embedding $\psi_2:A_\infty \to \ell_\infty$ such that $\psi_2(A_\infty)\cap c_0=\{0\}.$ Now let $A$ be an algebra or a Banach space, and $\phi:A\to \ell_\infty$ be an embedding. Then $T:=\psi_2\circ \psi_1^{-1}\circ  \phi$ is an $\ell_\infty$ embedding of $A$ that avoids $c_0.$
 \end{proof}

The first statement of \cref{trm l_inf in l_inf--c0} already appears implicitly in the proof of \cite[Theorem 1.3]{D.P}. The authors also note that a direct proof of this claim was suggested by the referee during the review process. As an alternative, we present a slightly different approach, constructing the copy inside a subalgebra of $\ell_\infty$ that already avoids $c_0$.
Consequently, 
 there is no need to verify that the copy $A_\infty$ of $\ell_{\infty}$ avoids $c_0$.

Before proceeding to the next section, it is worth noting that one must also be careful when deciding whether an algebra $A$ is universally $(\ell_{\infty}\setminus c_{0})$-embeddable, since ``being a compact element'' is relative to the ambient algebra. For instance, a closed subalgebra of $\ell_\infty$ may have nonzero compact elements even if it is disjoint from $c_0$.
This happens for the algebra $A_\infty=\phi(T(\ell_{\infty}))$ in the proof of \cref{trm l_inf in l_inf--c0}, where $A_0:=K(A_\infty)=\phi(T(c_0))\cong c_0$ and $A_0\cap c_0=\{0\}$ (we are using \cref{compact.operator}). Thus, neither $A_0$ nor $A_\infty$ is universally $(\ell_{\infty}\setminus c_{0})$-embeddable.

\subsection{$(B(\ell_2)\setminus \mathcal{K}(\ell_2))$-embeddability } \vspace*{10pt}

Now, let us recall a couple of definitions regarding C$^*$-algebras. For a given C$^*$-algebra $A,$ we write $A_{+}=\{a^{*}a: a\in A\}\subset A_{sa}=\{a\in A: a=a^{*}\}$ - here  $A_{+}$= the set of all \emph{positive} elements, and $A_{sa}$= the set of all \emph{self-adjoint elements}. In this setting an element $p\in A_{sa}$ satisfying  $p^{2}=p$ is called a \emph{projection}. Two projections $p,q\in A$ are \textit{orthogonal} if $pq=0.$
Finally, we say that a C$^*$-algebra $B$ is a \textit{von-Neumann algebra} if there exists a Banach space $B_{*}$ such that $B$ is the dual space of $B_{*}$ - this modern way of defining von-Neumann algebras is due to Sakai's theorem.

 A projection $p$ in a C$^*$-algebra $A$ is said to be \textit{minimal} if $pAp=\CC p.$ The \textit{socle} of $A$ ($\Soc(A)$) is the linear span of the set of minimal projections of $A.$ By \cite[Theorems 3.8 and 4.2]{Ylinen-compactCstar} (note that in \cite{Ylinen-compactCstar} minimal projections are called \textit{1-dimensional}), the ideal of compact elements of $A,$ $K(A),$ coincides with the closure of the socle of $A$. We denote by $\ell_2$ the complex Hilbert space $\ell_2(\NN)$  endowed with inner product $\langle f,g\rangle =\sum_n f(n) \overline{g(n)}.$ It is also well known that the mapping $\Phi\colon\ell_{\infty}\to B(\ell_2)$, defined by $f\mapsto \Phi(f)$ with $\Phi(f)(g)=fg$, is an injective unital $^*$-homomorphism (called the \textit{diagonal embedding}). Moreover, $\Phi(\ell_{\infty})$ is a MASA\footnote{it is not properly contained in any other abelian subalgebra of $B(\ell_2)$}  (\textit{maximal abelian subalgebra})   of $B(\ell_2)$. Before proceeding we want to recall that $K(B(\ell_{2}))=\mathcal{K}(\ell_{2}).$ 
Henceforth, throughout this section, $\Phi$ denotes the diagonal embedding $\Phi\colon \ell_{\infty}\to B(\ell_2)$ defined above.

 We claim that $K(\Phi(\ell_{\infty}))=\mathcal{K}(\ell_2)\cap \Phi(\ell_{\infty}).$ Indeed, if $p\in \Phi(\ell_{\infty})$ is a minimal projection, then by \cite[Lemma 3.3]{Kozminder} $p$ is also a minimal projection in $B(\ell_2)$, and hence $p\in \mathcal{K}(\ell_2)\cap \Phi(\ell_{\infty}).$ Thus $\Soc(\Phi(\ell_{\infty}))\subseteq \mathcal{K}(\ell_2)\cap \Phi(\ell_{\infty})$, and therefore $K(\Phi(\ell_{\infty}))\subseteq \mathcal{K}(\ell_2)\cap \Phi(\ell_{\infty}).$ The reverse inclusion is clear. This, together with \cref{compact.operator}, yields
 \begin{equation}\label{last?}
\Phi(c_0)=\Phi(K(\ell_\infty))=K(\Phi(\ell_\infty))=\mathcal{K}(\ell_2)\cap \Phi(\ell_\infty).
 \end{equation}

 Now, if $B \subset \ell_\infty$ satisfies $B\cap c_0=\{0\}$, then $\Phi(B)\cap K(\Phi(\ell_{\infty}))=\Phi(B)\cap \mathcal{K}(\ell_2)=\{0\}.$ Finally, we can compose $\Phi$ with the $(\ell_{\infty}\setminus c_{0}) $-embeddings obtained in \cref{trm l_inf in l_inf--c0}  to prove the next corollary. 

\begin{cor}\label{cor abelian triple in B(H)}
Every $\ell_\infty$-embeddable algebra and every $\ell_\infty$-embeddable Banach space is $(B(\ell_2)\setminus \mathcal{K}(\ell_2))$-embeddable. 
\end{cor}

\subsection{Obtaining some embeddibilities avoiding $c$ and $\mathcal{K}(\ell_2)\oplus \CC 1$}

All our results on $(\ell_{\infty}\setminus c_0)$-embeddability can also be transferred to the sets $(\ell_{\infty}\setminus c)$ and $B(\ell_2)\setminus\bigl(\mathcal{K}(\ell_2)\oplus \CC 1\bigr)$, where $\mathcal{K}(\ell_2)\oplus \CC 1$ denotes the compact perturbations of the identity operator (equivalently, the \textit{unitization} of $\mathcal{K}(\ell_2)$). Hereafter we denote $\mathcal{K}(\ell_2)\oplus \CC 1$ by $\widetilde{\mathcal{K}}(\ell_2)$.

\begin{thrm}\label{thrm l-c embedd} 
    Every $\ell_\infty$-embeddable algebra or Banach space is $(\ell_\infty\setminus c)$-embeddable. 
\end{thrm}
\begin{proof} 
Take $A_\infty=\phi(T(\ell_{\infty}))\cong\ell_{\infty}$ as in the proof of \cref{trm l_inf in l_inf--c0}, and let $\rho\colon A_{\infty}\to\ell_{\infty}$ be an $(\ell_\infty\setminus c_0)$-embedding, writing $\rho(a)=(\rho(a)_n)$. Define $\psi_2\colon A_\infty \to \ell_{\infty}$ by
$$
\psi_2(a)=\bigl(\rho(a)_1,0,\rho(a)_2,0,\rho(a)_3,0,\dots\bigr).
$$
Clearly, $\psi_2$ is an isometric $^*$-homomorphism.

We claim that $\psi_2(A_\infty)\cap c_0=\{0\}$. Indeed, let $b=(b_n)=\psi_2(a)$ and suppose that $\lim_n b_n=0$. Then $\lim_n \rho(a)_n=0$, so $\rho(a)\in c_0$, and hence $a=0$ because $\rho(A_{\infty})$ avoids $c_0$. Therefore $\psi_2(A_\infty)\cap c_0=\{0\}$.

Now suppose that $b=\psi_2(a)\in c$. Since $b_{2n}=0$ for all $n$ and $(b_n)$ converges, we must have $\lim_n b_n=0$, so $b\in c_0$ and hence $b=0$. Thus $\psi_2(A_\infty)\cap c=\{0\}$.

Next, let $A$ be an algebra or Banach space, let $\phi\colon A\to \ell_\infty$ be an embedding, and let $\psi_1\colon A_\infty \to \ell_\infty$ be a $^*$-isomorphism. Then
$$T:=\psi_2\circ \psi_1^{-1}\circ \phi$$
is an $\ell_\infty$-embedding of $A$ whose range avoids $c$.
\end{proof}

The existence of a unit is the main obstruction to universal $(\ell_{\infty}\setminus c)$-embeddability for infinite-dimensional uniform algebras. Indeed, if $\phi\colon A\to \ell_\infty$ is an embedding, then $\phi(1)=p=(p_n)$ is an idempotent ($p^2=p$) and $\phi(A)=\phi(A)p$. In particular, each $p_n\in\{0,1\}$.

If $p$ has finite support, then  $\phi(A)p$ is subalgebra of $\ell_{\infty}$ formed by just finitely many nonzero coordinates, and hence $A\cong \CC^n$, a contradiction. Thus $p$ has infinite support $S\subseteq\mathbb{N}$. Let $\sigma\colon S\to\mathbb{N}$ be a bijection, and define a new embedding $\psi\colon A\to\ell_\infty$ by
$$
\psi(a)=(\psi(a)_n)_{n\in\mathbb{N}},\qquad \psi(a)_n:=\phi(a)_{\sigma^{-1}(n)}.
$$
Then $\psi(1)=1_{\ell_\infty}\in c$. Hence, a unital infinite-dimensional uniform algebra is never universally $(\ell_{\infty}\setminus c)$-embeddable. This reflects the fact that $c$ is precisely the unitization of $c_0$. 

In the finite-dimensional case, this impossibility follows from \cref{main lineab banch algebra1} together with the fact that $K(A)=A$.

\begin{lem}\label{lem geral princpla K+C}
Let $B$ be a unital Banach algebra, and let $A$ be a Banach subalgebra of $B$.
Assume that $K(B)\neq \{0\}$ and $K(A)=\{0\}.$ Then:
\begin{enumerate}
\item \label{lem geral princpla K+C i}$A\cap K(B)=\{0\};$
\item \label{lem geral princpla K+C ii}if  $A\cap \bigl(K(B)\oplus \CC 1_B\bigr)\neq\{0\},$ then $A$ has a multiplicative identity.
\end{enumerate}
\end{lem}
\begin{proof}
\cref{lem geral princpla K+C i}Let us suppose that there exists $0\neq z\in A\cap K(B).$ Since $z\in K(B)$, so the  operator $U_z:B\to B$ is compact. Moreover, since $A$ is a (closed) subalgebra, we have $U_z(A)\subset A$, and hence the restriction $U_z\vert_A:A\to A$ is compact. 
Thus $K(A)\neq \{0\}$, a contradiction. It follows that $A\cap K(B)=\{0\}.$ 

\cref{lem geral princpla K+C ii}
For the remainder of the proof, set $M:=K(B)\oplus \CC 1_B$. Assume there exists $0\neq z\in A\cap M$ and write $z=k+\lambda 1_B$. If $\lambda=0$, then $0\neq z\in A\cap K(B)=\{0\}$, a contradiction; hence $\lambda\neq 0$.

 We claim that $w:=\frac{1}{\lambda}z=\frac{1}{\lambda}k+1_B\in A$ is a unit for $A$. Note that $w-1_B=\frac{1}{\lambda}k\in K(B)$. Since $A$ is a subalgebra and $K(B)$ is an ideal of $B$, for every $a\in A$ we have
 \begin{align*}
   aw-a &= a(w-1_B)\in A\cap K(B)=\{0\},\\
   wa-a &= (w-1_B)a\in A\cap K(B)=\{0\}.
 \end{align*}
Therefore $aw=wa=a$ for all $a\in A$, so $A$ contains a multiplicative identity (with unit $w\in A$). The proof is now complete.
\end{proof}

For our next result, recall that for any element $x=(x_n)\in c$, if $\lambda:=\lim_n x_n$, then $z=x-\lambda 1_{\infty}\in c_0$. Consequently, $x=z+\lambda 1_{\infty}$ for some $z\in c_0$. In particular, $c=K(\ell_\infty)\oplus \CC 1_\infty$.

\begin{cor}\label{cor char l-c univ embed} 
Let $A$ be an $\ell_\infty$-embeddable nonunital uniform algebra. $A$ is universally $(\ell_\infty\setminus c)$-embeddable if and only if $K(A)=\{0\}$.
\end{cor}
 \begin{proof}
 First observe that $A$ is  $(\ell_\infty\setminus c)$-embeddable by \cref{thrm l-c embedd}.
 
If $K(A)\neq\{0\}$, then there exists an $\ell_\infty$-embedding $\phi\colon A\to\ell_\infty$ with $\phi(A)\cap c_0\neq\{0\}$. In particular, $\phi(A)\cap c\neq\{0\}$.

Now assume that $K(A)=\{0\}$, and let $\phi\colon A\to\ell_\infty$ be an embedding. If $\phi(A)\cap\bigl(K(\ell_\infty)\oplus \CC 1_{\infty}\bigr)=\phi(A)\cap c\neq\{0\}$, then $\phi(A)$ has a multiplicative identity by \cref{lem geral princpla K+C}. Since $A\cong \phi(A)$, it follows that $A$ is unital, a contradiction.
  \end{proof}

Let $A=C(\Delta(A))$ be a unital and universally $(\ell_\infty\setminus c_0)$-embeddable C$^*$-algebra. Since $K(A)=\{0\}$, the space $\Delta(A)$ has no isolated points; hence $\Delta(A)\setminus \{x\}$ is locally compact and noncompact for every $x\in \Delta(A)$. Moreover, $\Delta(A)\setminus \{x\}$ has no isolated points either. Since every maximal ideal $I$ of $A$ is $^*$-isomorphic to $C_0(\Delta(A)\setminus \{x\})$ for some $x\in \Delta(A)$, we see that $K(I)=\{0\}$. Moreover, by \cref{cor char l-c univ embed}, every maximal ideal of $A$ is universally $(\ell_\infty\setminus c)$-embeddable. Thus, for any embedding $\phi\colon A\to\ell_\infty$, there are subalgebras (actually, ideals) of $A$ with co-dimension equal to $1$ that still avoid $c$.

\begin{cor}\label{cor B-(K+c) emebd} (See \cref{cor abelian triple in B(H)})
Every $\ell_\infty$-embeddable algebra and every $\ell_\infty$-embeddable Banach space is $\bigl(B(\ell_2)\setminus \widetilde{\mathcal{K}}(\ell_2)\bigr)$-embeddable.
\end{cor}

\begin{proof}
Let $X\subseteq\ell_{\infty}$ and $\Phi:\ell_\infty \to B(\ell_2)$ the diagonal embedding. We claim that
\begin{align}\label{eq intersct k+c}
\Phi(X)\cap\widetilde{\mathcal{K}}(\ell_2)=\Phi(X\cap c).  
\end{align}
Indeed, let $b\in X$ and assume that $\Phi(b)=k+\lambda I$ with $k\in \mathcal{K}(\ell_2)$ and $\lambda\in\CC$. Since $\Phi(1_{\ell_{\infty}})=I$, we obtain (see also \cref{last?})
$$
\Phi\bigl(b-\lambda 1_{\ell_{\infty}}\bigr)=k\in \mathcal{K}(\ell_2)\cap\Phi(\ell_{\infty})=\Phi(c_0).
$$
Hence $b-\lambda 1_{\ell_{\infty}}\in c_0$, and therefore $b\in c$. The reverse inclusion is immediate.

Now if $A$ is $\ell_{\infty}$-embeddable then by \cref{thrm l-c embedd} there exists an embedding  $\phi:A\to\ell_{\infty}$ that avoids $c.$ Let us define $\widetilde{\Phi}:=\Phi\circ\phi:A\to B(\ell_2).$ Then $\widetilde{\Phi}$ avoids $\widetilde{\mathcal{K}}(\ell_2)$ by \cref{eq intersct k+c}, since $\Phi(\phi(A))\cap\widetilde{\mathcal{K}}(\ell_2)=\Phi(\phi(A)\cap c)$.
\end{proof}

 We end this section with a sufficient condition for universal $\bigl(B(\ell_2)\setminus \widetilde{\mathcal{K}}(\ell_2)\bigr)$-embeddability. The proof is a direct application of \cref{lem geral princpla K+C} with $B=B(\ell_2).$ 

\begin{cor}\label{cor B-(K+c) univ emebd} 
Let $A$ be a $B(\ell_2)$-embeddable uniform algebra. If $K(A)=\{0\}$ then $A$ is universally  $\bigl(B(\ell_2)\setminus \mathcal{K}(\ell_2)\bigr)$-embeddable. Moreover, if $A$ is non-unital then $A$ is universally $\bigl(B(\ell_2)\setminus \widetilde{\mathcal{K}}(\ell_2)\bigr)$-embeddable.
\end{cor} 

\section{Algebrabilities and its variations}\label{section 3}

In this section, we study generators of abelian C$^*$-algebras with the aim of translating our embeddability results into precise statements on the algebrability of the sets considered in Section~\ref{section 2}.

\subsection{Generators}

\begin{defn}
Let $S$ be a nonempty set. We denote by $S^{<w}$ the set of \textit{finite sequences (or words) with entries in $S$}; equivalently,
$$
S^{<w}=\bigcup_{n\in \NN} S^n.
$$
Moreover, $|S^{<w}|=\max\{\aleph_0,|S|\}=\aleph_0\cdot|S|$.
\end{defn}

\begin{defn}
Let $A$ be a Banach algebra and let $\emptyset\neq S\subseteq A$. Set $A(S)=\Span(S^{<w})$, the (algebraic) subalgebra of $A$ generated by $S$, and $\Ban(S)=\overline{A(S)}$, the Banach subalgebra of $A$ generated by $S$. We say that $S$ generates $A$ as an algebra (respectively, as a Banach algebra) if $A=A(S)$ (respectively, $A=\Ban(S)$).
  \end{defn}

Given a subset $S$ of a C$^*$-algebra $A$, the C$^*$-subalgebra generated by $S$ is $C^*(S):=\Ban(S\cup S^*)$. In what follows, however, we adopt the convention (see the introduction of \cite{Nagisa}) of working only with generating sets consisting of self-adjoint elements, both for C$^*$-algebras and for self-adjoint (i.e., $*$-closed) subalgebras of C$^*$-algebras. Thus, when computing $\gn(A)$ or $\gn_{C^*}(A)$ (see \cref{rem Gen c* are sets}) in the sequel, we restrict to generating sets of self-adjoint elements.

\begin{defn}\label{classical definition}
Let $A$ be a C$^*$-algebra and let $S\subseteq A_{sa}$. We say that $S$ generates $A$ as a C$^*$-algebra if $A=\Ban(S)$. More generally, if $B$ is a self-adjoint subalgebra of $A$ and $S\subseteq B$, then $S$ generates $B$ as an algebra if $B=A(S)$.
\end{defn}
When no confusion can arise, we simply say that $S$ generates $A$ (or $B$), omitting the qualifiers ``as a C$^*$-algebra'' and ``as an algebra''.

\begin{rem} \label{rem Gen c* are sets}
Let $A$ be a commutative algebra. Denote by $\Gn(A)$, $\Gn_{\mathrm{ban}}(A)$, and $\Gn_{C^*}(A)$ the collections of all subsets of $A$ that generate $A$ as an algebra, as a Banach algebra, and as a C$^*$-algebra, respectively (whenever $A$ is endowed with the corresponding additional structure).
These collections are sets, by the axiom schema of specification applied to $\mathcal{P}(A)$. 

Define $\GC(A)=\{|S|: S\in \Gn(A)\}$.  
Note that $\GC(A)$ is the image of a well-defined function on the set $\Gn(A)$. Hence, by the axiom schema of replacement, $\GC(A)$ is a set of cardinal numbers. Similarly, we define the sets $\GC_{\mathrm{ban}}(A)$ and $\GC_{C^*}(A)$.
Additionally, by \cite[Theorem 23]{Kaplansky}, the sets $\GC(A)$, $\GC_{\mathrm{ban}}(A)$, and $\GC_{C^*}(A)$ are well ordered (with respect to the usual ordering of cardinals).   
Hence, we define $\gn(A)$ (respectively, $\gn_{\mathrm{ban}}(A)$ and $\gn_{C^*}(A)$) to be the least cardinality of a generating set of $A$ as an algebra (respectively, as a Banach algebra and as a C$^*$-algebra). 
 \end{rem}

\begin{rem}\label{rem non-self gen star}
Allowing generating sets that are not necessarily self-adjoint can reduce the required cardinality. More precisely, if $\gn_{C^*}(A)<\infty$, then the least cardinality of a generating set for $A$ as a C$^*$-algebra is $n$ when $\gn_{C^*}(A)=2n$, and $n+1$ when $\gn_{C^*}(A)=2n+1$ ($n\in\NN$). Indeed, if $S\in\Gn_{C^*}(A)$ with $S=\{s_1,\dots,s_{2n}\}\subset A_{sa}$, define $S':=\{s_j+\mathbf{i}s_{j+1}: j\in\{1,\dots,2n\}\text{ is odd}\}$. Since $(s_j+\mathbf{i}s_{j+1})^*+(s_j+\mathbf{i}s_{j+1})=2s_j$, the set $S'$ generates $A$, but $S'\nsubseteq A_{sa}$; hence $|S'|=n$, and this is the least cardinality of a system of generators of A
as a C$^*$-algebra. The case $|S|=2n+1$ is similar. Finally, for infinitely generated C$^*$-algebras, allowing non-self-adjoint generators does not change the least possible cardinality. The analogous statement for self-adjoint subalgebras $B$ of C$^*$-algebras with $\gn(B)<\infty$ also holds.
\end{rem}

For further discussion of this restriction, see \cref{prop Nagisa gener}. 

\begin{lem}\label{can not be finitely generated}
Let $A$ be an algebra. Then the following estimates hold:
\begin{enumerate}
\item\label{can not be finitely generated i} $\dim(A)\leq |S|\cdot \aleph_0=|S^{<w}|$ for every $S\in \Gn(A)$;
\item\label{can not be finitely generated ii} $\dim(A)\leq \gn(A)\cdot \aleph_0\leq \aleph_0\cdot \dim(A)$;
\item\label{can not be finitely generated iii} if $\dim(A)>\aleph_0$, then $\gn(A)=\dim(A)$.
\end{enumerate}
In particular, an infinite-dimensional Banach algebra cannot be finitely generated as an algebra.
\end{lem}
\begin{proof}
\cref{can not be finitely generated i} and \cref{can not be finitely generated ii}
Fix $S\in \Gn(A)$. If $\gamma$ is a Hamel basis of $A$, then $\gamma\in \Gn(A)$, and hence $\gn(A)\leq \dim(A)$; in particular, $\gn(A)\cdot\aleph_0\leq \dim(A)\cdot\aleph_0$.  On the other hand, since $S^{<w}$ spans $A$, the axiom of choice yields a Hamel basis $\beta\subset S^{<w}$. Thus $\dim(A)= |\beta|\leq |S^{<w}|=\aleph_0\cdot |S|$ for all $S\in \Gn(A)$, and therefore $\dim(A)\leq \gn(A)\cdot \aleph_0$. This completes the proof of $(i)$ and $(ii)$.

\cref{can not be finitely generated iii} Next, if $\dim(A)>\aleph_0$  then, by $(ii),$ we have $\aleph_0<\dim(A)\leq \gn(A)\cdot\aleph_0\leq \dim(A)\cdot \aleph_0.$ Thus, $\dim(A)=\gn(A)\cdot\aleph_0=\gn(A)=\dim(A).$

Finally, it is well known the dimension of every infinite dimensional Banach space is at least $\mathfrak{c}.$ Thus, if $A$ is an infinite dimensional Banach algebra then $\gn(A)=\dim(A)\geq \mathfrak{c}$ by \cref{can not be finitely generated iii}.
 \end{proof}

A Banach algebra with no nonzero compact elements is infinite dimensional. Indeed, if $A$ is finite dimensional, then for every $a\in A$ the operator $U_a\colon A\to A$ is compact. Thus, if $A$ is a Banach algebra with $K(A)=\{0\}$, then $\dim(A)=\infty$, and hence $A$ is not finitely generated as algebra by \cref{can not be finitely generated}.
 Consequently, an universally $(\ell_\infty\setminus c_0)$-embeddable Banach algebra $B$ is not finitely generated. Moreover, by \cref{can not be finitely generated} we have $\mathfrak{c}\leq \gn(B)=\dim(B)\leq \mathfrak{c}=\dim(\ell_\infty)$ and hence $\gn(B)=\mathfrak{c}.$ This conclusion also holds for infinite dimensional $\ell_\infty$-embeddable commutative Banach algebras. We use this as a fact along the rest of the paper without explicitly mentioning it.

\begin{rem}\label{countable-generating}
     It is not hard to see that a C$^*$-algebra $A$ is separable if and only if $\gn_{C^*}(A)\leq \aleph_0$. Indeed, if $A$ is separable we may find a countable dense subset $X$ of $A.$ Consider $S=\{(x+x^{*}), \mathbf{i}(x-x^{*}): x\in X\}\subset A_{sa}.$ Note that $x=\frac{1}{2}(x+x^*)-\frac{\mathbf{i}}{2}(\mathbf{i}(x-x^*).$ Hence, $S\in \Gn_{C^*}(A)$ and $|S|=\aleph_{0}.$ Then, $\gn_{C^*}(A)\leq \aleph_0.$ Conversely,  if $\gn_{C^*}(A)\leq \aleph_0,$ we may find $S\subset A_{sa}$ where $|S|=\aleph_{0}$ and $\overline{A(S)}=A.$  The set $S^{<w}$ is countable. Consequently, 
     $\{\sum_{j=1}^n (\alpha_j +\beta_j\mathbf{i})x_j: n\in \NN,\alpha_j,\beta_j\in \mathbb{Q},x_j\in S^{<w}\}$
      is a countable dense subset of $A.$
     
\end{rem}

\subsection{Calculating \texorpdfstring{$\gn_{C^{*}}(A)$}{gnC*(A)} in terms of certain topological invariants}

The next result characterizes $\gn_{C^{*}}(A)$ for finitely generated C$^*$-algebras.
\begin{prop}\label{prop Nagisa gener}
    (\cite[Proposition~2]{Nagisa}) Let $A$ be a unital commutative C$^{*}$-algebra with spectrum $\Delta(A)$. Then
    \begin{equation*}
        \gn_{C^*}(A)=\Min\{m\in\mathbb{N} : \ \mbox{there is an embedding of} \ \Delta(A) \ \mbox{into} \ \mathbb{R}^{m} \}.
    \end{equation*}
\end{prop}

Let $A$ be a unital commutative C$^*$-algebra. We note that \cite{Nagisa} uses the notation $\gn(A)$ for our $\gn_{C^*}(A)$; in this paragraph we follow that convention.
Moreover, \cite{Nagisa} adopts the convention $\gn(A)=\infty$ when $A$ is not finitely generated. Our next goal is to obtain an appropriate version of \cref{prop Nagisa gener} for infinitely generated C$^*$-algebras, which will allow us to compute $\gn(A)$ and gives a clearer picture of the case $\gn(A)=\infty$.
From this point on, we revert to our standing convention from \cref{rem Gen c* are sets}.

\begin{thrm}\label{teo genCstar}
Let $A$ be a unital commutative C$^*$-algebra. Then,  the following equality holds: 
  \begin{equation*}
  \GC_{C^*}(A)=\{ \kappa\leq |C(\Delta(A))| : \kappa \  \mbox{is a cardinal, and} \ \exists  \Phi: \Delta(A) \to \mathbb{R}^{\kappa} \ \mbox{embedding}\}.
  \end{equation*}
In particular, 
 \begin{equation*}
        \gn_{C^*}(A)=\Min\{ \kappa\leq |C(\Delta(A))| : \kappa \  \mbox{is a cardinal, and} \ \exists  \Phi: \Delta(A) \to \mathbb{R}^\kappa \mbox{ embedding}\}.
    \end{equation*}
\end{thrm}
\begin{proof}
If $A$ is finitely generated, then statement follows from \cref{prop Nagisa gener}. Thus, let us assume that $A$ is not finitely generated.

Let $S=(a_{\lambda})_{\lambda\in \Lambda}\in \Gn_{C^{*}}(A)$ be a set of generators of $A.$         By \cite[Theorem 3.1.10]{Rickart} the mapping
 $$\Phi_S: \Delta(A)\to \CC^{\Lambda},\varphi\mapsto (\hat{a_{\lambda}}(\varphi))$$ is a homeomorphism onto its image, that is, an embedding. The mapping $\Phi_S$ is said to be the \textit{canonical representation of $\Delta(A)$  induced by $S.$} Recall that $\Delta(A)$ is compact in this case.
 
 Next, let $\kappa$ be an infinite cardinal number with $\kappa\leq |C(\Delta(A))|.$  Let $X$ be a set with $|X|=\kappa,$ and for each $\lambda \in X$ we define  $f_{\lambda}:\Delta(A)\to \CC $ by $f_{\lambda}=\pi_{\lb} \circ\Phi,$ where $\Phi:  \Delta(A) \to \CC^{\kappa}$ is an embedding,  and $\pi_{\alpha}$ is the canonical projection. Observe that we cannot rule out the possibility that for some $\varphi\in \Delta(A)$ we have  $f_{\lambda}(\varphi)=0$ for all $\lambda \in X.$ For this reason we consider the family $\tilde{S}=\{f_{\lambda}: \lambda\in X\}\cup\{1\}.$ Thus the family $\tilde{S}$ vanishes nowhere, and $\tilde{S}\subset C(\Delta(A)).$
 Moreover, if $\varphi_1\neq\varphi_2\in\Delta(A)$ then, by the injectivity of $\Phi$, $\Phi(\varphi_1)\neq\Phi(\varphi_2)$, so there exists $\lambda\in X$ such that $f_\lambda(\varphi_1)=\pi_\lambda(\Phi(\varphi_1))\neq\pi_\lambda(\Phi(\varphi_2))=f_\lambda(\varphi_2)$. 
 So, $\tilde{S}$ separates points in $\Delta(A)$. Set $\tilde{S}_{sa}=\{(g+g^{*}), \mathbf{i}(g-g^{*}): g\in \tilde{S}\}\subset C(\Delta(A))_{sa}$ (see also \cref{countable-generating}).  By the Stone-Weierstrass Theorem $\tilde{S}_{sa}$ generates $C(\Delta(A))$ as a C$^*$-algebra. Applying the inverse of the Gelfand transform (which is a $^*$-isomorphism) we may infer that there exist $S\subset A_{sa}$ where $S\in \Gn_{C^*}(A)$ with $|S|\leq \kappa,$ since $|X|=\kappa.$ In case $|S|<\kappa $ we claim that there exists $T \in \Gn_{C^*}(A)$ such that $S\subseteq T$ and $|T|=\kappa.$ Indeed, 
$\kappa\leq|C(\Delta(A))|=|C(\Delta(A))_{sa}|.$ So, there exists a $S'\subseteq C(\Delta(A))_{sa}$ with $|S'|=\kappa$ and $T:=S\cup S_1$ ($S_1$=the inverse image of $S'$ by the Gelfand transform) satisfies the desired properties. 
 
 On the other hand, take an infinite set $S$ with $|S|=\kappa.$ We note that $|S|=|S\times \{0,1\}|=|S|$. Thus, there exists a bijective function $f:S\to S\times \{0,1\}.$ If we set $S_i=f^{-1}(S\times\{i\})$ then $S=S_{0}\dot\cup S_{1}.$ Thus we can apply \cite[Proposition 2.3.7]{Engelking} with $X_s=\mathbb{R}$ for all $s\in S,$  $T=\{0,1\}$ and $S_t,t\in \{0,1\}$ the partition of $S$ constructed above. More precisely the following homeomorphism holds
 (here the topology to be considered is the classical product topology):
 $$\mathbb{R}^\kappa=\Pi_{s\in S}\mathbb{R}\cong\Pi_{t\in \{0,1\}}(\Pi_{s\in S_t} \mathbb{R})=\mathbb{R}^\kappa\times \mathbb{R}^\kappa .$$ A straightforward computation allows us to show that $\CC^\kappa\cong \mathbb{R}^\kappa\ \times \mathbb{R}^\kappa$ (homeomorphically).

 Thus, we have seen that if  $A$ is a unital commutative C$^*$-algebra that is not finitely generated as a C$^*$-algebra, then the set $\GC_{C^*}(A)$ can be described in the following way: 
  \begin{equation*}
  \GC_{C^*}(A)=\{ \kappa\leq |C(\Delta(A))| : \kappa \  \mbox{is a cardinal, and} \ \exists  \Phi: \Delta(A) \to \CC^{\kappa} \ \mbox{embedding}\}.
\end{equation*}
By composing with an homeomorphism from $\CC^\kappa$ onto $\mathbb{R}^\kappa$ we can easily see that $$\GC_{C^*}(A)=\{ \kappa\leq |C(\Delta(A))| : \kappa \  \mbox{is a cardinal, and} \ \exists  \Phi: \Delta(A) \to \mathbb{R}^\kappa \mbox{ embedding}\}.$$ The proof is now complete.
\end{proof}

  Based on \cref{teo genCstar} we introduce the following definition: 

 \begin{defn}\label{definition}
   Let $K$ be a compact Hausdorff space.
   We define the \emph{set of embedding cardinals of} $K$ by
   \begin{equation*}
  \EC(K)= \{\kappa\leq |C(K)|: \kappa\ \mbox{is a cardinal and  exists an embedding }\Phi: K \to \mathbb{R}^{\kappa}\}.
    \end{equation*}
 \end{defn}

For the moment, we recall a couple of definitions concerning some important cardinal invariants of a topological space. We then show that $\gn_{C^*}(A)$ can be computed in terms of these invariants (\cref{prop num gen C*}).

\begin{defn}
  Let $X=(X,\tau)$ be a topological space. The \emph{density character} of $X$, denoted by $d(X)$, is the least cardinality of a dense subset of $X$. The \emph{weight} of $X$, denoted by $w(X)$, is the least cardinality of an open base for $(X,\tau)$.
\end{defn}

 \begin{prop}\label{prop num gen C*}
     Let $A$ be an abelian C$^*$-algebra that is not finitely generated as a C$^*$-algebra. Then 
     $$ w(\Delta(A))=d(A)=\gn_{C*}(A)=\min(\EC( L_A)),$$ where $L_A=\Delta(A)$ if $A$ is unital and $L_A=\omega\Delta(A)$ otherwise\footnote{the one-point compactification of $\Delta(A)$}.
 \end{prop}
 \begin{proof}
First we assume that $A$ is unital; equivalently, $A\cong C(K)$, where $K=\Delta(A)$ is a compact Hausdorff space. The equality $\gn_{C*}(A)=\min(\EC( L_A))$ follows from \cref{teo genCstar}.

   By \cite[Lemma 2.4]{HolNov}   we have $d(C(K))=w(K).$  Take $S\in \Gn_{C^*}(A).$  Since $S$ is infinite $|S^{<w}|=|S|$. Now, we set
   \begin{equation*}
     H(S)=\{\sum_{j=1}^n (\alpha_j +\mathbf{i}\beta_j)x_j: n\in \NN,\alpha_j,\beta_j\in \mathbb{Q},x_j\in S^{<w}\}.
     \end{equation*}
     Then $|H(S)|=|S|$ and $H(S)$ is dense in $A.$ Thus $d(A)\leq \gn_{C*}(A).$ By the proof of \cite[Lemma 2.4]{HolNov} there exists $S_1\subseteq A$ such that $|S_1|=w(K)$, where $S_1$ generates $A_{sa}$ as a (real) Banach algebra ($A_{sa}$= the set of all self-adjoint elements of $A$). Therefore, it is clear, since $A=A_{sa}+iA_{sa},$ that $S_1$ generates $A$ as a C$^*$-algebra and hence $\gn_{C*}(A)\leq |S_1|=w(K)=d(A).$ This finishes the unital case.

Assume now that $A$ is a non-unital C$^*$-algebra.
Then $A\cong C_0(L)$, where $L=\Delta(A)$ is locally compact Hausdorff. It is well known that $C_0(L)$ is $^*$-isomorphic to $\{f\in C(\omega L): f(0)=0\}$, where $\omega L=L\cup\{0\}$ is the one-point compactification of $L$. Moreover,
$$
C(\omega L)\cong C_0(L)\oplus \CC 1 \cong A\oplus \CC 1,
$$
where $A\oplus \CC 1$ is the unitization of $A$. By \cite[Lemma 1]{Nagisa} we have $\gn_{C^*}(A\oplus \CC 1)=\gn_{C^*}(A)$, and therefore \cref{teo genCstar} yields
$$
\min(\EC(\omega L))=\gn_{C^*}(C(\omega L))=\gn_{C^*}(A\oplus \CC 1)=\gn_{C^*}(A).
$$
Since $A\neq 0$, we also have $\aleph_0\le d(A)=d(C_0(L))=d(C(\omega L))=d(A\oplus \CC 1)$. Moreover, \cite[Theorem 3.5.11]{Engelking} shows that $w(L)=w(\omega L)$ for every locally compact Hausdorff space $L$. The conclusion now follows from the unital case.

\end{proof}

Now let $\kappa$ be a cardinal and set $K_\kappa=[0,1]^\kappa$. Then $K_\kappa$ is compact, Hausdorff, and perfect (i.e., it has no isolated points). Moreover, $K_\kappa$ is separable for every $\kappa\leq \mathfrak{c}$ by \cite[Corollary 2.3.16]{Engelking}. We can now prove the following result.

\begin{cor}\label{cor genC C(kappa)}
   $\gn_{C^*}(C(K_\kappa))= \kappa$ for every cardinal $\kappa.$  
\end{cor}
\begin{proof}
If  $\kappa$ is infinite, then $w([0,1]^\kappa)=\kappa$ 
(see the comment after the proof of  \cite[Theorem 2.3.23]{Engelking}) 
and the conclusion follows from  \cref{prop num gen C*}. 

If $\kappa=n<\infty$, then clearly $\gn_{C^*}(C(K_n))\le n$ by \cref{prop Nagisa gener}. Suppose, towards a contradiction, that there exists an embedding $\Phi:K_n \to \mathbb{R}^m$ for some $m<n$. 
Let $U\subseteq K_n$ be a nonempty open set; then the restriction $\Phi_{|U}$ is an embedding of $U$ into $\mathbb{R}^m$, contradicting 
the topological invariance of dimension's Theorem. Thus $\gn_{C^*}(C(K_n) )\geq n$ again by \cref{prop Nagisa gener}.
\end{proof}

\subsection{Algebrability vs genalgebrability of \texorpdfstring{$(\ell_{\infty}\setminus c_0)\cup\{0\}$}{(l-infty\textbackslash{}c0)\textbackslash{}{0}}}

\begin{defn}
    Let $A$ be a Banach algebra and let $S\in \Gn(A)$ (respectively, $S_1\in \Gn_{\mathrm{ban}}(A)$). We say that $S$ (respectively, $S_1$) is a \emph{minimal system of generators} (or \emph{minimal generating set}) for $A$ if no proper subset of it generates $A$ as an algebra (respectively, as a Banach algebra). Similarly, if $A$ is a C$^{*}$-algebra, we call $S_2\in \Gn_{C^*}(A)$ \emph{minimal} if no proper subset of $S_2$ generates $A$ as a C$^*$-algebra. 
\end{defn}

The definition below can be found in \cite[Definition V.3]{A.B.P.S}.

\begin{defn}\label{def classical algebrable}
   Let $A$ be a Banach algebra and $\alpha,\beta$ be  cardinal numbers. We say that a set $M\subset A$ is $(\alpha,\beta)$-\emph{algebrable} (respectively, algebrable) if there exists a subalgebra $B$ of $A$ such that $B\subseteq M \cup \{0\},$ $\dim(B)=\alpha,$ $|S|=\beta$ where $S$ is a minimal  system  of generators of $B$ \textit{as an algebra} (respectively, the cardinality of any system of generators of $B$ is infinite). 
 \end{defn}

 Regarding the specific set $$(\ell_{\infty}\setminus c_0)\cup\{0\},$$ its spaceability was proved by Rosenthal in the 1960s \cite{H.P.R.1,H.P.R.2}, as a consequence of his result that $c_0$ is quasi-complemented in $\ell_{\infty}$ \cite{H.P.R.1}. Concerning algebrability, Garc\'ia-Pacheco, Mart\'in and Seoane-Sep\'ulveda \cite{G.M.S} showed (among other results) that $(\ell_{\infty}\setminus c_0)\cup\{0\}$ contains subalgebras with an infinite minimal set of generators; more recently Papathanasiou \cite{D.P} strengthened the picture by proving strong $\mathfrak{c}$-algebrability.

The techniques developed in this section allow us to control the number of generators of an abelian C$^*$-algebra. For this reason, we consider the following variant of the classical definitions of algebrability.

\begin{defn}\label{def genalgebrable}
   Let $A$ be a Banach algebra and $d,n$  cardinal numbers. We say that a set $M\subset A$ is $\kappa$-\emph{genalgebrable} (respectively, $(d,\kappa)$-genalgebrable) if there exists a subalgebra $B$ of $A$ such that $B\subseteq M \cup \{0\}$ and $\gn(B)=\kappa$ (respectively, $\dim(B)=d$ and $\gn(B)=\kappa$). If  in addition $n=\dim(A),$ then we say that $M$ is maximal genalgebrable.
 \end{defn}

\begin{defn}\label{defn.alg.}
  Let $A$ be a C$^*$-algebra and $n$ a cardinal number. We say that a set $M\subset A$ is $\kappa$-C$^*$-\emph{genalgebrable} if there exists a C$^*$-subalgebra $B$ of $A$ such that $B\subseteq M\cup\{0\}$ and $\gn_{C^*}(B)=\kappa$. If, in addition, $\kappa=\dim A$, then we say that $M$ is maximal-C$^*$-genalgebrable. Similarly, $(d,\kappa)$-C$^*$-genalgebrability can be defined (see \cref{def genalgebrable}).
  \end{defn}
\begin{rem}  
It is worth noting that $\gn(A)$ is always well-defined, and \cref{def genalgebrable} (respectively \cref{defn.alg.}) does not rely on the existence of a minimal generating set. 
\end{rem}

\cref{cor genC C(kappa)} together with \cref{trm l_inf in l_inf--c0} already implies that $(\ell_{\infty}\setminus c_0)\cup \{0\}$ is $\kappa$-C$^*$-\emph{genalgebrable} for every $\kappa\leq \mathfrak{c}$. In the next result, we also keep track of the dimensions of the algebras in order to identify all pairs $(d,\kappa)$ for which $(\ell_{\infty}\setminus c_0)\cup \{0\}$ is $(d,\kappa)$-C$^*$-genalgebrable.
In doing so, we obtain a complete classification of the pairs $(d,\kappa)$ for which $(d,\kappa)$-(C$^*$)-genalgebrability holds.

\begin{thrm} \label{teo n-alg and n-ctar-alg}
Let $M=(\ell_{\infty}\setminus c_0)\cup \{0\}$. Then:
\begin{enumerate}
    \item \label{teo n-alg and n-ctar-alg i} $M$ is $(d,1)$-genalgebrable and $(d,1)$-C$^*$-genalgebrable for every $d\in \mathbb{N}$.
    \item \label{teo n-alg and n-ctar-alg ii} If $d,n<\infty$ and $n\neq 1$, then $M$ is neither $(d,n)$-genalgebrable nor $(d,n)$-C$^*$-genalgebrable.
    \item \label{teo n-alg and n-ctar-alg iii} $M$ is $(\mathfrak{c},\kappa)$-C$^*$-genalgebrable for every $\kappa\leq \mathfrak{c}$. 
    \item \label{teo n-alg and n-ctar-alg iv} $M$ is $(\aleph_0,\kappa)$-genalgebrable for every $\kappa\leq \aleph_0$.
    \item \label{teo n-alg and n-ctar-alg v} $M$ is $(\kappa,\kappa)$-genalgebrable for every $\aleph_0 < \kappa \leq \mathfrak{c}$. 
    \end{enumerate}
\end{thrm}

\begin{proof}
Set $M=   
 (\ell_{\infty}\setminus c_0)\cup \{0\}.$  By \cref{trm l_inf in l_inf--c0}, we first build copies of the algebras with the desired properties in $\ell_\infty$ and then transfer them to $M.$ Note that the dimension, $\gn$, and $\gn_{C^*}$ are preserved under embeddings. We also consider non-closed algebras; these are subalgebras of some $\ell_\infty$-embeddable $C(K)$ space, so the results transfer as well. Throughout, we take $C(K)$-spaces with $K$ separable, so $\ell_\infty$-embeddability follows from \cref{lem linf-embedable}.  
Throughout the proof, $(p_n)\subset \ell_\infty$ denotes a sequence of mutually orthogonal projections.
 
$\cref{teo n-alg and n-ctar-alg i}$ 
For any fixed $d\in \mathbb{N}$ let us consider the algebra $A_d=\Span\{p_i:i\in 1\leq i\leq d\}.$ 
The equality $\gn(A_d)=\gn_{C^*}(A_d)=1$ is well known, and it also follows directly from our results. Indeed, $A_d\cong \CC^d\cong C(\{1,2,\ldots,d\})$. Since $\{1,\dots,d\}$ embeds into $\mathbb{R}$, \cref{prop Nagisa gener} yields $\gn_{C^*}(A_d)=1$. As $d<\infty$, we have $A_d(S)=\overline{A_d(S)}$ for every $S\subseteq A_d$, and hence $\gn(A_d)=\gn_{C^*}(A_d)=1$. Therefore, for every $d\in \mathbb{N}$, the set $M$ is $(d,1)$-genalgebrable and $(d,1)$-C$^*$-genalgebrable.

 $\cref{teo n-alg and n-ctar-alg ii}$
 Observe that every closed subalgebra $C$ of an abelian C$^*$-algebra $A$ is semisimple; in particular, every closed subalgebra of $\ell_{\infty}$ is semisimple. Hence, if $C$ is a finite-dimensional subalgebra of $\ell_{\infty}$, then $C$ is a commutative semisimple Banach algebra. By the Artin--Wedderburn theorem, $C\cong \CC^{\dim(C)}$, and therefore $\gn(C)=\gn_{C^*}(C)=1$ by \cref{teo n-alg and n-ctar-alg i}.
 
\cref{teo n-alg and n-ctar-alg iii}.
For every $\kappa\leq \mathfrak{c}$ we have the equalities $\dim(C(K_\kappa))=2^{\aleph_0}=\mathfrak{c}$ (because $K_\kappa$ is separable)  and $\gn_{C^*}(C(K_\kappa))=\kappa$ (see \cref{cor genC C(kappa)} and the preceding comments).
  Thus $M$ is $(\mathfrak{c},\kappa)$-C$^*$-genalgebrable for $\kappa \leq \mathfrak{c}$. So, we have \cref{teo n-alg and n-ctar-alg iii}. 

  \cref{teo n-alg and n-ctar-alg iv} and \cref{teo n-alg and n-ctar-alg v}. 
Now fix $\kappa\leq \mathfrak{c}$. Choose $S_\kappa\in \Gn_{C^*}(C(K_\kappa))$ with $|S_\kappa|=\kappa$ and set $A_\kappa:=A(S_\kappa)$ ($A_\kappa$ is a self-adjoint subalgebra). We claim that $\gn(A_\kappa)=\kappa$.  Indeed, $\gn(A_\kappa)\leq \kappa$ because $S_\kappa$ generates $A_\kappa$ and $|S_\kappa|=\kappa$. Conversely, if $\gn(A_\kappa)=\kappa'<\kappa$, then there exists $T\in \Gn(A_\kappa)$ with $|T|=\kappa'$, and hence $A_\kappa=A(T)$. Taking closures in $C(K_\kappa)$ yields $C(K_\kappa)=\overline{A_\kappa}=\overline{A(T)}$, so $T\in \Gn_{C^*}(C(K_\kappa))$, contradicting $|T|<\kappa=\gn_{C^*}(C(K_\kappa))$. Therefore, $\gn(A_\kappa)=\kappa$.
 
It remains to compute $\dim(A_{\kappa})$. First, note that $\dim(A_{\kappa})\geq \aleph_0$, since otherwise $A_{\kappa}$ would be finite-dimensional and therefore so would its closure $\overline{A_{\kappa}}=C(K_\kappa)$.
 
 By \cref{can not be finitely generated} we have 
 $$\aleph_0\leq \dim(A_\kappa) \leq \gn(A_\kappa) \cdot \aleph_0.$$
Thus if $\kappa\leq \aleph_0$ we have $\dim(A_\kappa)=\aleph_0$ and $\gn(A_\kappa)=\kappa$ (proving \cref{teo n-alg and n-ctar-alg iv}). 

Finally, if $\kappa>\aleph_0$ then $\kappa=\gn(A_\kappa)\cdot \aleph_0 \leq \aleph_0 \dim(A_\kappa)$ by \cref{can not be finitely generated} and hence $\dim(A_\kappa)=\kappa.$ 
So, $\gn(A_\kappa)=\dim(A_\kappa)=\kappa$ yielding \cref{teo n-alg and n-ctar-alg v}.
\end{proof} 

In what follows, let $\alpha$ and $\beta$ be cardinals. \cref{teo n-alg and n-ctar-alg} covers all ``admissible'' pairs $(\alpha,\beta)$ for $(\alpha,\beta)$-(C$^*$)-genalgebrability of $M=(\ell_{\infty}\setminus c_0)\cup\{0\}$. More precisely, any pair that does not appear in its statement is genuinely impossible: there is no algebra (respectively, no C$^*$-algebra) realizing the corresponding values of $(\alpha,\beta)$, not merely no such subalgebra of $M$ (see also \cref{impossibility}).

The phenomenon witnessed in \cref{teo n-alg and n-ctar-alg} may also seem counterintuitive: the only obstruction to $(\alpha,\beta)$-(C$^*$)-genalgebrability occurs when both $\alpha$ and $\beta$ are finite and $\beta\neq 1$, which yields a kind of discontinuity in the behavior of $(\alpha,\beta)$-(C$^*$)-genalgebrability. This discontinuity disappears when one studies algebrability in the classical sense, precisely because one can find minimal generating sets with different cardinalities. Therefore, the notions of C$^*$-genalgebrability and genalgebrability may be viewed as refinements of algebrability, providing a hierarchy between mere nonemptiness and the existence of richer algebraic structure.

\begin{rem}\label{impossibility}
Observe that, by \cref{can not be finitely generated}, there is no algebra $A$ with $\dim(A)=\mathfrak{c}$ admitting a countable generating set $S\in\Gn(A)$ (that is, $|S|=\aleph_0$). On the other hand, there is no C$^*$-algebra $A$ with $\dim(A)=\aleph_0$ either. 

It is straightforward to check that every other pair $(\alpha,\beta)$ not covered by \cref{teo n-alg and n-ctar-alg} is indeed ``impossible'' (in the sense explained above).
\end{rem}

We introduce the following definition, which appears to be missing from the literature:
\begin{defn}\label{missing}
    Let $A$ be a C$^*$-algebra and $d,n$   cardinal numbers.  We  say that a subset $M$ of  $A$ is \textit{$(d,\kappa)$-C$^*$-algebrable} if there exists a C$^*$-subalgebra $B$ of $A$ such that $B\subseteq M\cup \{0\},$ $d=\dim(B)$ and there exists $S\in \Gn_{C^*}(B)$ such that $|S|=\kappa$ and $S$ is a minimal generating set (as a C$^*$-algebra). 
\end{defn}

\begin{rem}\label{rem min gen set if gec*<inf?}
If $A$ is a C$^*$-algebra with $\gn_{C^*}(A)<\infty$, then there exists a minimal generating set $S$ with $|S|=\gn_{C^*}(A)$. However, minimal generating sets need not all have the same cardinality. For instance $c_0$ is a singly generated C$^*$-algebra  but also has a minimal generating set formed by a sequence of mutually orthogonal projections. On the other hand, when $\gn_{C^*}(A)=\infty$, we do not even know whether a minimal generating set exists. This means that obtaining $(d,\kappa)$-C$^*$-algebrability relies on constructing particular examples that have a minimal generating set (when possible). Although we are able to find such constructions for all possible cardinalities that appear in \cref{teo n-alg and n-ctar-alg}, it is important to notice that for some well-known algebras, such as $\ell_\infty,$ it seems to be unknown whether a minimal generating set (as an algebra or as a C$^*$-algebra) exists at all. Thus, in general, genalgebrability results cannnot be transferred \textit{automatically} to algebrability results.
\end{rem}

\begin{lem}\label{lem kappa minimal generator}
For every cardinal $\kappa$, the C$^*$-algebra $B_\kappa:=C(K_\kappa)$ admits a minimal generating set $S_\kappa\in \Gn_{C^*}(B_\kappa)$ with $|S_\kappa|=\kappa$. Moreover, $S_\kappa$ is a minimal generating set for $A(S_\kappa)$.
\end{lem}
\begin{proof}
For $\kappa<\infty$, the existence of $S_\kappa$ follows directly from the definition of $\gn_{C^*}$: in the finite case, removing an element from a generating set strictly decreases its cardinality.

Recall that, by our convention, $S_\kappa\subseteq (B_\kappa)_{sa}$. Consequently, $A_\kappa:=A(S_\kappa)$ is a self-adjoint subalgebra of $B_\kappa$, so when working in $A_\kappa$ we restrict attention to generating sets consisting of self-adjoint elements (even for infinite cardinals).
Now, fix $f\in S_\kappa$. If $A_\kappa=A(S_\kappa\setminus\{f\})$, then $C(K_\kappa)=\overline{A_\kappa}=\overline{A(S_\kappa\setminus\{f\})}$, so $S_\kappa\setminus\{f\}\in\Gn_{C^*}(C(K_\kappa))$ and $|S_\kappa\setminus\{f\}|<\kappa=\gn_{C^*}(C(K_\kappa))$, a contradiction. Therefore, $S_\kappa$ is a minimal generating set for $A_\kappa$.

Next, fix $\kappa$ infinite. Let $X$ be a set where $|X|=\kappa.$
For each $i\in X,$
%$i<\kappa,$ 
let $\pi_i:K_\kappa \to [0,1]$ be the canonical projection. First observe that for each $i$ we have $\pi_i\in C(K_\kappa)_{sa}$ since $\pi_i(K_\kappa)=[0,1].$ It follows that the set $S_\kappa=\{ \pi_i:i\in X\}\cup \{1\}$ is selfadjoint. It is clear that  $S_\kappa$ separates points and vanishes nowhere. By the Stone-Weierstrass theorem $S_\kappa\in \gn_{C^*}(C(K_\kappa)).$ In addition, $|S_{\kappa}|=\kappa.$

We now show that $S_\kappa$ is in fact a minimal generating set. Fix $i\in X$ and set $S_{\kappa,i}:=S_\kappa\setminus\{\pi_i\}$.
%We claim that $S_{\kappa,i}$ does not separate points in $K_\kappa$. Indeed, 
Take $x,y\in[0,1]^\kappa$ such that $x_i\neq y_i$ and $x_j=y_j$ for all $j\neq i$. Then $x\neq y$ but $f(x)=f(y)$ for every $f\in S_{\kappa,i}$, and hence $\overline{A(S_{\kappa,i})}\neq C(K_\kappa)$ since $C(K_{\kappa})$ separates points and $g(x)=g(y)$ for all $g\in \overline{A(S_{\kappa,i})}.$
Moreover, $S_0:=S_\kappa\setminus\{1\}$ is not a generating set either, since every function in $A(S_0)$ vanishes at $(0,\dots,0)\in[0,1]^\kappa$. Therefore $S_\kappa\in\Gn_{C^*}(C(K_\kappa))$ is minimal, and $|S_\kappa|=\kappa$.

Finally, note that $A_{\kappa,i}:=A(S_{\kappa,i})\subseteq A_\kappa:=A(S_\kappa)$. Since $A_\kappa$ separates points whereas $A_{\kappa,i}$ does not, we have $A_{\kappa,i}\neq A_\kappa$. Likewise, $S_0$ cannot generate $A_\kappa$ because $1\notin A(S_0)$. This shows that $S_\kappa$ is a minimal generating set for $A_\kappa$.
\end{proof}

Our next goal is to determine all possible $(\alpha,\beta)$-(C$^{*}$)-algebrability phenomena for $(\ell_{\infty}\setminus c_0)\cup\{0\}$. Although transferring algebrability results to the (C$^{*}$)-algebrability-setting is not automatic (see \cref{rem min gen set if gec*<inf?}), we were able to exhibit minimal generating sets for all algebras used in \cref{teo n-alg and n-ctar-alg}.

\begin{cor} \label{2-teo n-alg and n-ctar-alg}
Let $M=(\ell_{\infty}\setminus c_0)\cup \{0\}$. Then:
\begin{enumerate}
    \item \label{2teo n-alg and n-ctar-alg i} $M$ is $(d,n)$-algebrable and $(d,m)$-C$^*$-algebrable for all $d\in \mathbb{N}$ with $d\geq m,n.$ 
    \item \label{2teo n-alg and n-ctar-alg ii} $M$ is $(\mathfrak{c},\kappa)$-C$^*$-algebrable for all $\kappa\leq \mathfrak{c}$.
    \item \label{2teo n-alg and n-ctar-alg iii} $M$  is  $(\aleph_0,\kappa)$-algebrable for $\kappa \leq \aleph_0$.
    \item \label{2teo n-alg and n-ctar-alg iv} $M$  is  $(\kappa,\kappa)$-algebrable for $\aleph_0< \kappa \leq \mathfrak{c}$.
    \end{enumerate}
    \end{cor}
\begin{proof}
Throughout the proof, the objects $A_d$, $A_\kappa$, $C(K_\kappa)$, and $(p_n)$ have the same meaning as in the proof of \cref{teo n-alg and n-ctar-alg}. We follow the same strategy: we work first in $\ell_\infty$ and then transfer the conclusions to $M$.

\cref{2teo n-alg and n-ctar-alg i} 
Assume that, for each $d\in \NN$, the algebra $\CC^d$ admits a minimal set of generators as an algebra (equivalently, in this case, as a C$^*$-algebra) of cardinality $n$ for every $n\in \{1,\dots,d\}$. Then the same conclusion holds for $A_d$, exactly as in the proof of \cref{teo n-alg and n-ctar-alg}. The case $n=1$ is treated there.

Now fix $k\in \{2,\dots,d\}$ and let $p_1,\dots,p_d\in \CC^d=C(\{1,2,\ldots,d\})$ be mutually orthogonal projections. Write $\CC^d=A\oplus B$, where $A=\Span(\{p_1,\dots,p_k\})$ and $B=\Span(\{p_{k+1},\dots,p_d\}).$ Let $g$ be a generator for $A$. Then $\{g,p_{k+1},\dots,p_d\}$ is a minimal set of generators (since it consists of mutually orthogonal elements) of cardinality $(d-k)+1$. Note that $\{d-k+1:2\leq k\leq d\}=\{1,\dots,d-1\}$. Thus, for each $d\in \NN$ we can find a subalgebra of $\ell_{\infty}$ with a minimal set of generators (also as a C$^*$-algebra) of cardinality $n$ for all $n<d$. The case $n=d$ is clear since $A_d$ is spanned by $\{p_1,\dots,p_d\}$. Hence, $M$ is $(d,n)$-(C$^{*}$)-algebrable for all $d,n\in\mathbb{N}$ with $n\leq d$.

 \cref{2teo n-alg and n-ctar-alg ii} Follows from \cref{teo n-alg and n-ctar-alg}\cref{teo n-alg and n-ctar-alg iii} (and its proof),  \cref{lem kappa minimal generator} and the fact that  $\dim(C(K_\kappa))=\mathfrak{c}$ for all $\kappa\leq \mathfrak{c}.$

\cref{2teo n-alg and n-ctar-alg iii} and \cref{2teo n-alg and n-ctar-alg iv} follow from, respectively, \cref{teo n-alg and n-ctar-alg}\cref{teo n-alg and n-ctar-alg iv} and \cref{teo n-alg and n-ctar-alg v} (and their proofs) together with \cref{lem kappa minimal generator}.  
\end{proof}
Finally, by applying \cref{thrm l-c embedd} or \cref{cor abelian triple in B(H)}, every affirmative case in \cref{teo n-alg and n-ctar-alg} and \cref{2-teo n-alg and n-ctar-alg} transfers to $(\ell_\infty\setminus c)\cup\{0\}$ and $(B(\ell_2)\setminus \mathcal{K}(\ell_2))\cup\{0\}$.
Moreover, using \cref{thrm l-c embedd} together with \cref{cor B-(K+c) emebd}, the same conclusions hold for $(\ell_\infty\setminus c)\cup\{0\}$ and $(B(\ell_2)\setminus \widetilde{\mathcal{K}}(\ell_2))\cup\{0\}$.

\begin{cor}\label{cor c*algebrable}
The sets
\[
(B(\ell_2)\setminus \mathcal{K}(\ell_2))\cup\{0\},\qquad (\ell_{\infty}\setminus c)\cup\{0\},\qquad (B(\ell_2)\setminus \widetilde{\mathcal{K}}(\ell_2))\cup\{0\}
\]
are $(d,n)$-genalgebrable and $(d,n)$-C$^*$-genalgebrable (respectively, $(d,n)$-algebrable and $(d,n)$-C$^*$-algebrable) for every $(d,n)$ for which \cref{teo n-alg and n-ctar-alg} (respectively, \cref{2-teo n-alg and n-ctar-alg}) gives an affirmative answer.
 \end{cor} 

\textbf{Acknowledgments}
 Jorge J. Garcés was partially supported by grant PID2021-122126NB-C31 funded by MCIN/AEI/10.13039/501100011033 and by ERDF/EU and Junta de Andalucía grant FQM375. W. Franca was partially supported by FAPEMIG grant APQ-02575-24, and IMU grant 1057. Both authors were partially supported by the IMAG--Mar\'ia de Maeztu grant CEX2020-001105-M/AEI/\allowbreak 10.13039/\allowbreak 501100011033.\smallskip

 Part of this work was completed during several research visits: first to the Polytechnic university of Madrid (November 2023), and second to the Federal University of Juiz de Fora (August 2025). A significant portion of the work was also conducted during a joint visit by both authors to IMAG at the University of Granada (June 2026), and a subsequent visit by the second author to the SRMC at Peking University. Both authors thank these institutions for their hospitality and support. We thank the anonymous reviewer for their valuable comments and suggestions, which have significantly improved this manuscript.\medskip

\bibliographystyle{plain}
\bibliography{bibl}{}
\end{document}